\documentclass[a4paper,10pt]{amsart}

\usepackage{microtype}
\usepackage{amsthm,amsmath,amssymb,graphicx}
\usepackage{graphicx,subcaption,enumerate}
\usepackage{tikz}

\usetikzlibrary{arrows,backgrounds,decorations,automata}
\newtheorem{theorem}{Theorem}

\newtheorem{proposition}[theorem]{Proposition}

\newtheorem{corollary}[theorem]{Corollary}
\newtheorem*{remark}{Remark}

\theoremstyle{definition}
\newtheorem{definition}[theorem]{Definition}

\makeatletter
\DeclareRobustCommand\bigop[2][1]{%
  \mathop{\vphantom{\sum}\mathpalette\bigop@{{#1}{#2}}}\slimits@
}
\newcommand{\bigop@}[2]{\bigop@@#1#2}
\newcommand{\bigop@@}[3]{%
  \vcenter{%
    \sbox\z@{$#1\sum$}%
    \hbox{\resizebox{\ifx#1\displaystyle#2\fi\dimexpr\ht\z@+\dp\z@}{!}{$\m@th#3$}}%
  }%
}
\makeatother

\newcommand{\ee}{\mathrm{e}}

\begin{document}

\title[Ghost distributions of regular sequences are self-affine]{Ghost distributions of regular sequences are affine transformations of self-affine sets}



\author{Michael Coons}
\author{James Evans}
\author{Zachary Groth}
\author{Neil Ma\~{n}ibo}

\address{School of Information and Physical Sciences, 
   University of Newcastle, \newline
\hspace*{\parindent}University Drive, Callaghan NSW 2308, Australia}
\email{michael.coons@newcastle.edu.au, james.evans10@uon.edu.au,\newline \hspace*{\parindent}zachary.groth@uon.edu.au}

\address{Fakult\"at f\"ur Mathematik, Universit\"at Bielefeld, \newline
\hspace*{\parindent}Postfach 100131, 33501 Bielefeld, Germany}
\email{cmanibo@math.uni-bielefeld.de }

\date{\today}

\begin{abstract} Ghost measures of regular sequences---the unbounded  analogue of automatic sequences---are generalisations of standard fractal mass distributions. They were introduced to determine fractal (or self-similar) properties of regular sequences similar to those related to automatic sequences. The existence and continuity of ghost measures for a large class of regular sequences was recently given by Coons, Evans and Ma\~nibo. In this paper, we provide an explicit connection between fractals and regular sequences by showing that the graphs of ghost distributions---the distribution functions of ghost measures---of the above-mentioned class of regular sequences are sections of self-affine sets. As an application of our result, we show that the ghost distributions of the Zaremba sequences---regular sequences of the denominators of the convergents of badly approximable numbers---are all singular continuous. 
\end{abstract}

\maketitle

\section{Introduction}

Symbolic dynamics, fractal geometry, number theory and theoretical computer science share many connections, but arguably none so pervasive as their focus on understanding patterns, and specifically their desire to understand how relatively simple procedures---e.g., doubling, recurrence, iteration---can take the most mundane objects and produce rich structures with long range order. An interval becomes the Cantor set, two tiles arrange to give Penrose's tiling and a pair of simple functions lead to a strange attractor. While several notions and concepts have been introduced to describe the structure and symmetry of these limiting objects, the most successful of these has been the mathematical theory of measures and the related theory of (fractal) dimension.  

Central objects in this context, ubiquitous to each of the above-mentioned areas, are uniform substitutions---more commonly studied in arithmetic areas under the name `automatic sequences'. A sequence $f$ is \emph{$k$-automatic} provided there is a deterministic finite automaton that takes in the base-$k$ expansion of a positive integer $n$ and outputs the value $f(n)$. Automatic sequences can be described in many ways; the one most appropriate for our current purposes is via the $k$-kernel, $${\rm ker}_k ({f}):=\left\{(f(k^\ell n+r))_{n\geqslant 0}: \ell\geqslant 0, 0\leqslant r<k^\ell\right\}.$$ A sequence $f$ is $k$-automatic if and only if its $k$-kernel is finite \cite[Prop.~V.3.3]{E1974}. It is immediate that an automatic sequence takes only a finite number of values. A natural generalisation to sequences that can be unbounded was given in the early nineties by Allouche and Shallit \cite{AS1992}; an integer-valued sequence $f$ is called \emph{$k$-regular} if the $\mathbb{Q}$-vector space $\mathcal{V}_k(f):=\langle{\rm ker}_k ({f})\rangle_\mathbb{Q}$ generated by the $k$-kernel of $f$ is finite-dimensional over $\mathbb{Q}$. One nice property of this generalisation is that a bounded regular sequence is automatic. Additionally, the set of $k$-regular sequences has algebraic structure, it forms a ring under point-wise addition and convolution. 

Let $k\geqslant 2$ be an integer, $f$ be a $k$-regular sequence and let the set of integer sequences $\{f=f_1,f_2,\ldots,f_d\}\subseteq \mathcal{V}_k(f)$ be a basis for the $\mathbb{Q}$-vector space $\mathcal{V}_k(f)$. Set ${\bf f}(m)=(f_1(m),f_2(m),\ldots,f_d(m))$ and for each $a\in\{0,\ldots,k-1\}$ let ${\bf B}_a$ be the $d\times d$ integer matrix such that, for all $m\geqslant 0$, $${\bf f}(km+a)={\bf f}(m){\bf B}_a.$$ We call the ${\bf B}_a$ {\em digit matrices} and write $\mathcal{B}:=\{{\bf B}_0,{\bf B}_1,\ldots,{\bf B}_{k-1}\}$. Since the $d\times d$ digit matrices of a regular sequence are taken from a basis, the positive integer $d$ is minimal; we call a regular sequence {\em $d$-dimensional} provided its digit matrices are $d\times d$ matrices. See the seminal paper of Allouche and Shallit \cite{AS1992} and Nishioka's monograph \cite[Ch.~15]{N1996} for details on existence and the finer definitions. It follows from above that there is ${\bf w}\in\mathbb{Z}^{1\times d}$ such that for each $i\in\{1,\ldots,d\}$ and $n>0$, we have \begin{equation}\label{eq:linear}f_i(m)={\bf w}\, {\bf B}_{(m)_k} e_i={\bf w}\, {\bf B}_{i_s}\cdots{\bf B}_{i_1} {\bf B}_{i_0}\, e_i,\end{equation} where $e_i$ is the $i$th elementary column vector, $(m)_k=i_s\cdots i_1i_0$ is the base-$k$ expansion of $m$ and ${\bf B}_{(m)_k}:={\bf B}_{i_s}\cdots{\bf B}_{i_1} {\bf B}_{i_0}$. Set ${\bf B}:=\sum_{a=0}^{k-1}{\bf B}_a$. 

Recently, Coons, Evans and Ma\~nibo \cite{CEM}, in a first attempt to view $\mathcal{V}_k(f)$ dynamically, used an analogy between the matrix ${\bf B}$ and the substitution matrix of a substitution on a finite alphabet to prove the existence of a probability measure associated to $f$, called a {\em ghost measure}; see also \cite{CE2021,Epre}. They noted that, ``In this way, the existence of a natural measure associated to $f$ and $\mathcal{V}_k(f)$ provides a path to viewing the space $\mathcal{V}_k(f)$ of regular sequences associated to $f$ as a dynamical system and opens the possibility of considering these sequences and spaces as (fractal) geometric structures.'' In this paper, we show this possibility is reality, by proving that these ghost distributions---the distribution functions of ghost measures---are affine transformations of self-affine sets. 

In analogy with primitive substitutions, it is reasonable to restrict ourselves to the assumption that ${\bf B}$ is primitive. But, if this is the case for the $k$-regular sequence $f$, then we can consider $f$ as a $k^j$-regular sequence with $j$ being the smallest positive integer for which ${\bf B}^j$ is positive. Thus, in the context of $k$-regular sequences the distinction between primitivity and positivity is somewhat blurred. This is tacitly done for substitutions, where one normally chooses an appropriate power $j$ such that ${\bf M}_{\varrho}^{j}>0$, or equivalently, $\varrho^{j}(a)$ contains all the letters of the alphabet $\mathcal{A}$, for all $a\in \mathcal{A}$. Hence we make the following definition.

\begin{definition}\label{def:prim} We call a $k$-regular sequence $f$ {\em primitive} provided $f$ takes non-negative integer values, is not eventually zero, each of the $k$ digit matrices are non-negative and the matrix ${\bf B}$ is positive.
\end{definition}

To state our result formally, we introduce some notation. Set \begin{equation*}\label{eq:Sigmai}\varSigma(n):=\sum_{m=k^n}^{k^{n+1}-1}f(m)\quad\mbox{and}\quad
   \mu^{}_{n} \, := \, \frac{1}{\varSigma(n)} \sum_{m=0}^{k^{n+1}-k^n - 1}
   f(k^n + m) \, \delta^{}_{m / k^n(k-1)},
\end{equation*}
where $\delta_x$ denotes the unit Dirac measure at $x$. We can view
$(\mu^{}_{n})^{}_{n\in\mathbb{N}_0}$ as a sequence of probability measures on
the $1$-torus, the latter written as $\mathbb{T}=[0,1)$ with addition modulo
$1$. Here, we have simply re-interpreted the (normalised) values of 
the sequence $(f(m))_{m\geqslant 0}$ between $k^n$ and $k^{n+1}-1$ as the weights of a pure point
probability measure on $\mathbb{T}$ supported on the set $\big\{ {m}/{(k^n(k-1))} : 0 \leqslant m <
k^n(k-1) \big\}$. 

With the above notation, Coons, Evans and Ma\~nibo \cite[Thm.~2]{CEM} proved that if $f$ is a positive integer-valued $k$-regular sequence, such that the spectral radius of ${\bf B}$ is the unique dominant eigenvalue of ${\bf B}$, then the pure point probability measures $\mu_{n}$ converge weakly to a continuous probability measure $\mu_f$ on $\mathbb{T}$. The measure $\mu_f$ is called the {\em ghost measure} of $f$, and its distribution function is called the {\em ghost distribution} of $f$. 

As stated above, our main result connects the ghost distribution of a regular sequence to a self-affine set. Recall that an {\em affine contraction} $S:\mathbb{R}^n\to\mathbb{R}^n$ is a transformation of the form $S(x)=T(x)+b$, where $T$ is a linear contraction on $\mathbb{R}^n$ representable as an $n\times n$ matrix and $b\in\mathbb{R}^n$. A finite family of affine contractions $\mathcal{S}=\{S_1,\ldots,S_m\}$, with $m\geqslant 2$, is called an {\em iterated function system}, and the unique attractor (or invariant set) for an iterated function system of affine contractions is called a {\em self-affine set}; see Falconer \cite[Ch.~9]{Fbook} for details on iterated function systems over $\mathbb{R}^n$.

\begin{theorem}\label{thm:ghostselfaffine} If $f$ is a primitive $k$-regular sequence, then the graph of the ghost distribution of $f$ is an affine transformation of a self-affine set.
\end{theorem}

This paper is organised as follows. In the next section, with the proof of Theorem \ref{thm:ghostselfaffine} as a goal, we provide an explicit relationship between the solution of a dilation equation and the attractor of an iterated function system of affine contractions (Theorem \ref{thm:dilationtoifs}). In that section, we use a classical geometric construction of a singular continuous function due to Salem \cite{S1943} as a motivating example. Section \ref{sec:relations} concludes with a general result (Theorem \ref{main}) that implies Theorem \ref{thm:ghostselfaffine}. In Section \ref{sec:1d}, we consider `Salem sequences' as generalisations of standard missing digit sequences and provide a complete characterisation of their ghost measures; these sequences can be viewed as weighted generalisations of the standard middle-thirds Cantor set---their ghost distributions are generalisations of the Devil's staircase. In Section \ref{sec:zaremba}, we apply our results to so-called `Zaremba sequences', the $2$-dimensional $k$-regular sequences that encode the denominators of badly approximable numbers with $k$ given partial quotients. In particular, we show that any nontrivial Zaremba ghost measure is singular continuous with respect to Lebesgue measure; that is, its distribution function has almost everywhere zero derivative. Finally, in the Concluding Remarks, we consider some subtleties of our approach and contrast a bit the Salem versus Zaremba sequences.

\section{Dilation equations, self-affine sets and ghost distributions}\label{sec:relations}

In this section, we will prove Theorem \ref{thm:ghostselfaffine}, by providing a more general result (Theorem \ref{main}). To do this we provide an explicit relationship between the solution of a dilation equation and the attractor of an iterated function system of affine contractions. Since the growth properties of regular sequences are related to the norms of their digit matrices, before moving on to these tasks, we define the joint spectral radius of a finite set of matrices. 

Throughout this paper, we let $\rho({\bf M})$ denote the spectral radius of the matrix ${\bf M}$ and denote the {\em joint spectral radius} of a finite set of matrices $\mathcal{M}:=\{{\bf M}_1,{\bf M}_2,\ldots, {\bf M}_{\ell}\}$ by $$\rho^*(\mathcal{M})=\limsup_{n\to\infty}\max_{1\leqslant i_0,i_1,\ldots,i_{n-1}\leqslant \ell}\left\| {\bf M}_{i_0}{\bf M}_{i_1}\cdots{\bf M}_{i_{n-1}}\right\|^{1/n},$$
where $\|\cdot\|$ is any submultiplicative matrix norm---here, we will exclusively use the standard operator norm. This quantity was introduced by Rota and Strang \cite{RS1960} and has a wide range of applications. For an extensive treatment, see Jungers \cite{J2009}.

Throughout this paper, we assume that the spectral radius $\rho=\rho({\bf B})$ of ${\bf B}$ is the unique dominant eigenvalue of ${\bf B}$ (by `unique' we implicitly mean `simple' as well); for discussion of a more general treatment see Coons, Evans and Ma\~nibo \cite[Sec.~3]{CEM}. Let ${\bf v}_\rho$ is the $d\times 1$ eigenvector associated to~$\rho$. 
We define the vector-valued function ${\bf F}:\mathbb{R}\to\mathbb{R}^{d}$ by \begin{equation}\label{dilation}{\bf F}(x) = \sum_{a=0}^{k-1}\rho^{-1}\, {\bf B}_{a} \cdot {\bf F}(k x-a),\quad\mbox{where}\quad{\bf F}(x)=\begin{cases}{\bf 0} & \text{ for } x \leqslant 0\\ {\bf v}_{\rho} &\text{ for } x \geqslant 1.\end{cases}\end{equation} Functional equations such as \eqref{dilation} are known as \emph{dilation equations} or \emph{two-scale difference equations} in the literature; seminal work on these was done by Daubechies and Lagarias \cite{DL1991,DL1992} and their relationship to regular sequences was developed and highlighted by Dumas \cite{D2013,D2014}.  The function ${\bf F}$ exists and is unique if $\rho({\bf B})>\rho^*(\mathcal{B})$. Moreover, the function ${\bf F}$ is H\"older continuous with exponent $\alpha$ for any $\alpha<\log_k(\rho/\rho^*).$

\begin{theorem}\label{thm:dilationtoifs} Let $k\geqslant 2$ be an integer and $f$ be a $k$-regular sequence such that the spectral radius $\rho=\rho({\bf B})$ is the unique dominant eigenvalue of ${\bf B}$ and that $\rho({\bf B})>\rho^*(\mathcal{B}).$ Let ${\bf F}(x)$ be the solution of the dilation equation \eqref{dilation}. Then the graph $$\mathcal{F}_f:=\{(x,{\bf F}(x)):x\in[0,1]\}\subset[0,1]^{d+1}$$ of ${\bf F}(x)$ is a self-affine set. In particular, $\mathcal{F}_f$ is the attractor of the iterated function system $\mathcal{S}_f=\{S_0,\ldots,S_{k-1}\}:[0,1]^{d+1}\to[0,1]^{d+1}$ where, for $j\in\{0,1,\ldots,k-1\}$, we have $$S_j\left(\begin{matrix}y_0\\ \vdots\\ y_{d}\end{matrix}\right)=\left(\begin{matrix} 1/k & {\bf 0}^{1\times (k-1)}\\ {\bf 0}^{(k-1)\times 1} & \rho^{-1}\, {\bf B}_j\end{matrix}\right)\left(\begin{matrix}y_0\\ \vdots\\ y_{d}\end{matrix}\right)+\left(\begin{matrix}j/k\\  \sum_{a=0}^{j-1}\rho^{-1}\, {\bf B}_a{\bf v}_\rho\end{matrix}\right).$$
\end{theorem}

\begin{remark} While it is certainly known that one can iteratively construct solutions of dilation equations, to the best of our knowledge, until now a general result like Theorem \ref{thm:dilationtoifs} has not appeared in the literature. Note, also, that Theorem \ref{thm:dilationtoifs} is a two-way correspondence, providing a way to go back and forth between iterated function systems of contractions and solutions of dilation equations.
\end{remark}

Before proving Theorem \ref{thm:dilationtoifs}, and to entice a little why one may be interested in such a relationship, we give an example of a regular sequence, and its related dilation equation which has an explicit connection to a self-affine set considered by Salem. 

In the early 1940s, Salem \cite{S1943} gave a geometric construction of (the graphs of) a family of strictly increasing singular continuous functions from $[0,1]$ to $[0,1]$. Recall that a function is singular continuous provided it has almost everywhere zero derivative. Here we consider Salem's example, with parameter $\lambda_0=2/5$. In this case, we denote Salem's self-affine set by $\mathcal{A}_s$, which is the unique attractor of the iterated function system $\mathcal{S}_s=\{S_0,S_1\}:[0,1]^2\to[0,1]^2$, where $$S_0\left(\begin{matrix} x\\ y\end{matrix}\right)=\left(\begin{matrix} 1/2 & 0\\ 0 &2/5\end{matrix}\right)\left(\begin{matrix} x\\ y\end{matrix}\right)\quad\mbox{and}\quad S_1\left(\begin{matrix} x\\ y\end{matrix}\right)=\left(\begin{matrix} 1/2 & 0\\ 0 &3/5\end{matrix}\right)\left(\begin{matrix} x\\ y\end{matrix}\right)+\left(\begin{matrix} 1/2 \\ 2/5\end{matrix}\right).$$ Figure \ref{fig:Salem} shows how the attractor $\mathcal{A}_s$ is developed by iterating the system $\mathcal{S}_s$ starting with the line segment joining $(0,0)$ and $(1,1)$ as a seed.

\begin{figure}[ht]
\begin{center}
  \includegraphics[width=.32\linewidth]{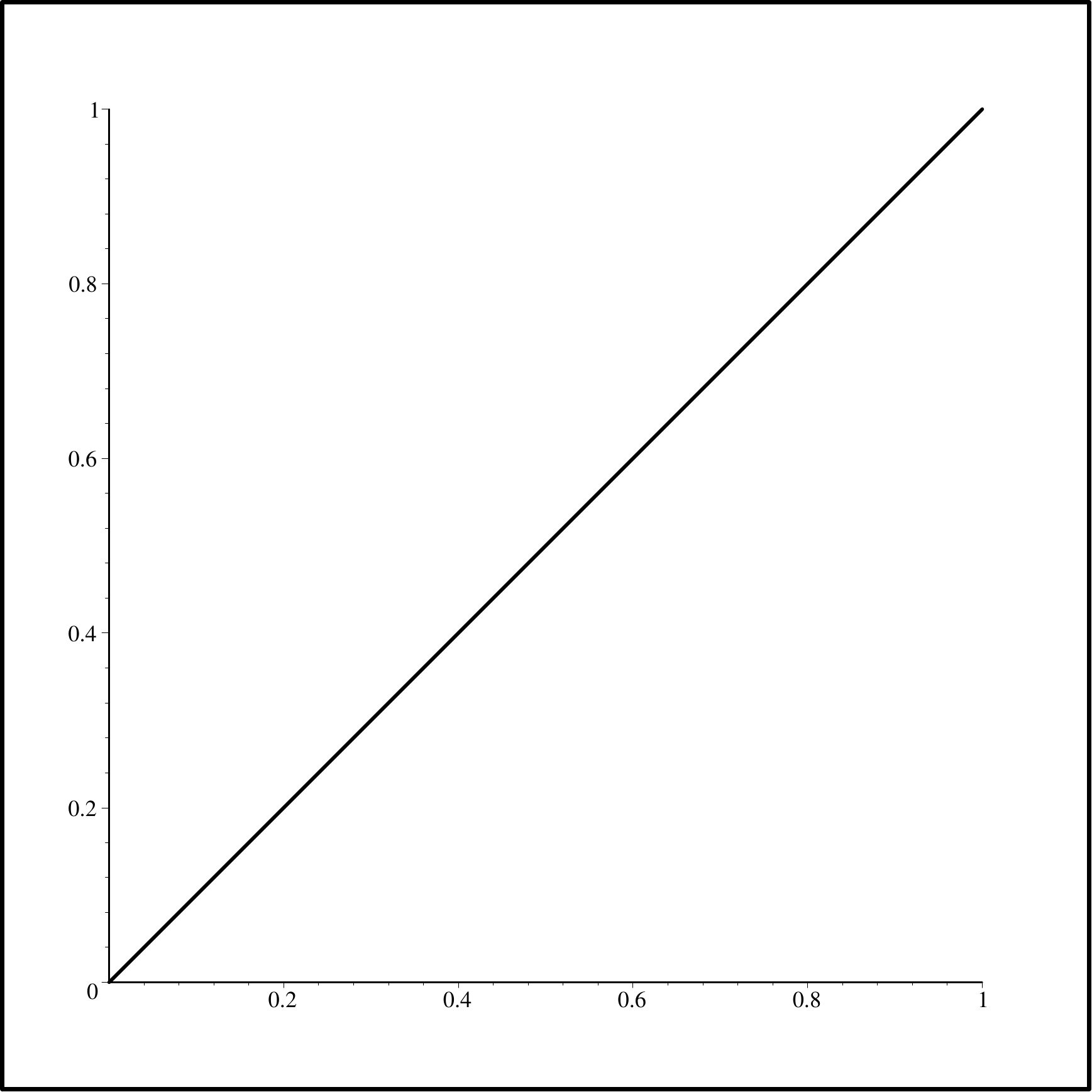}
  \includegraphics[width=.32\linewidth]{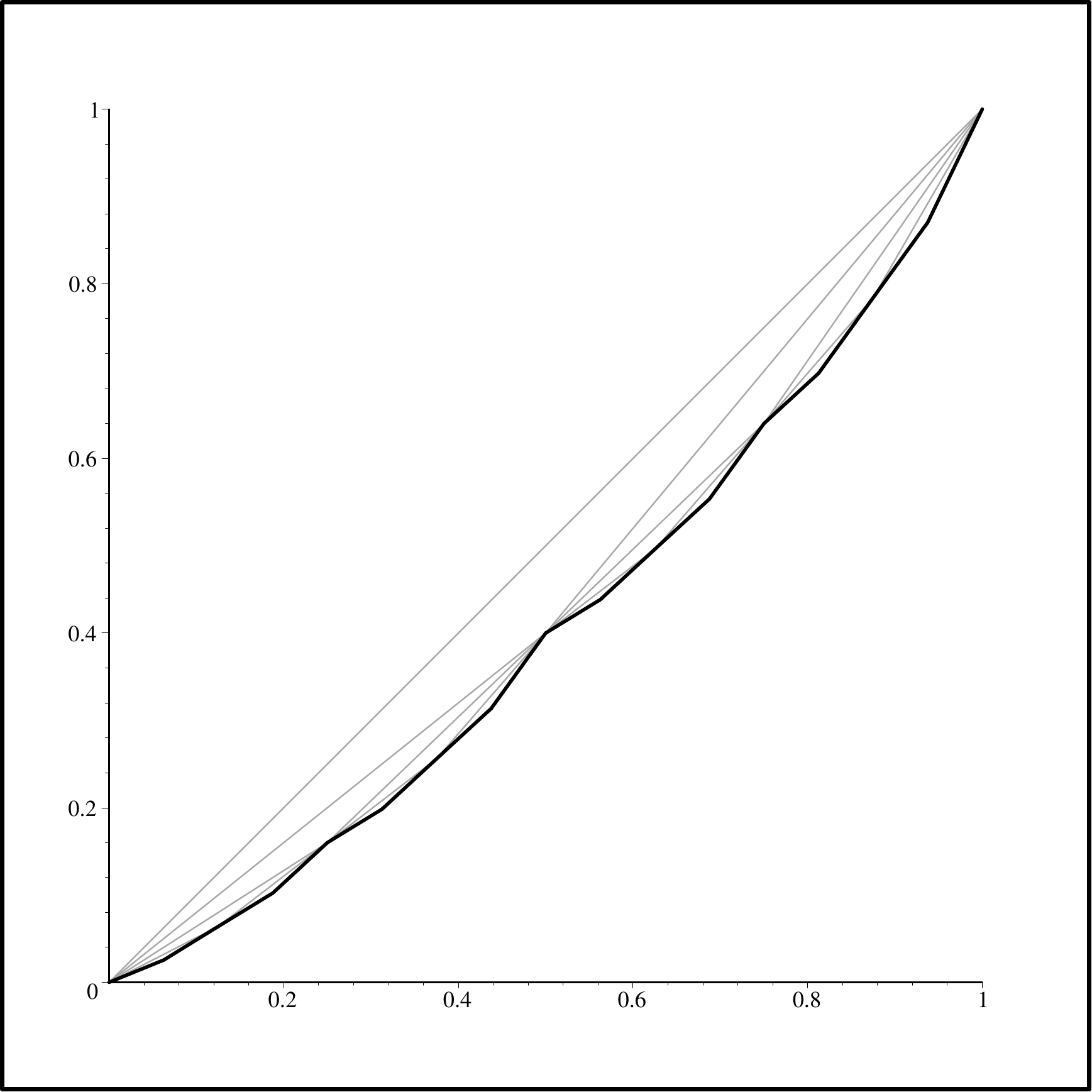} 
  \includegraphics[width=.32\linewidth]{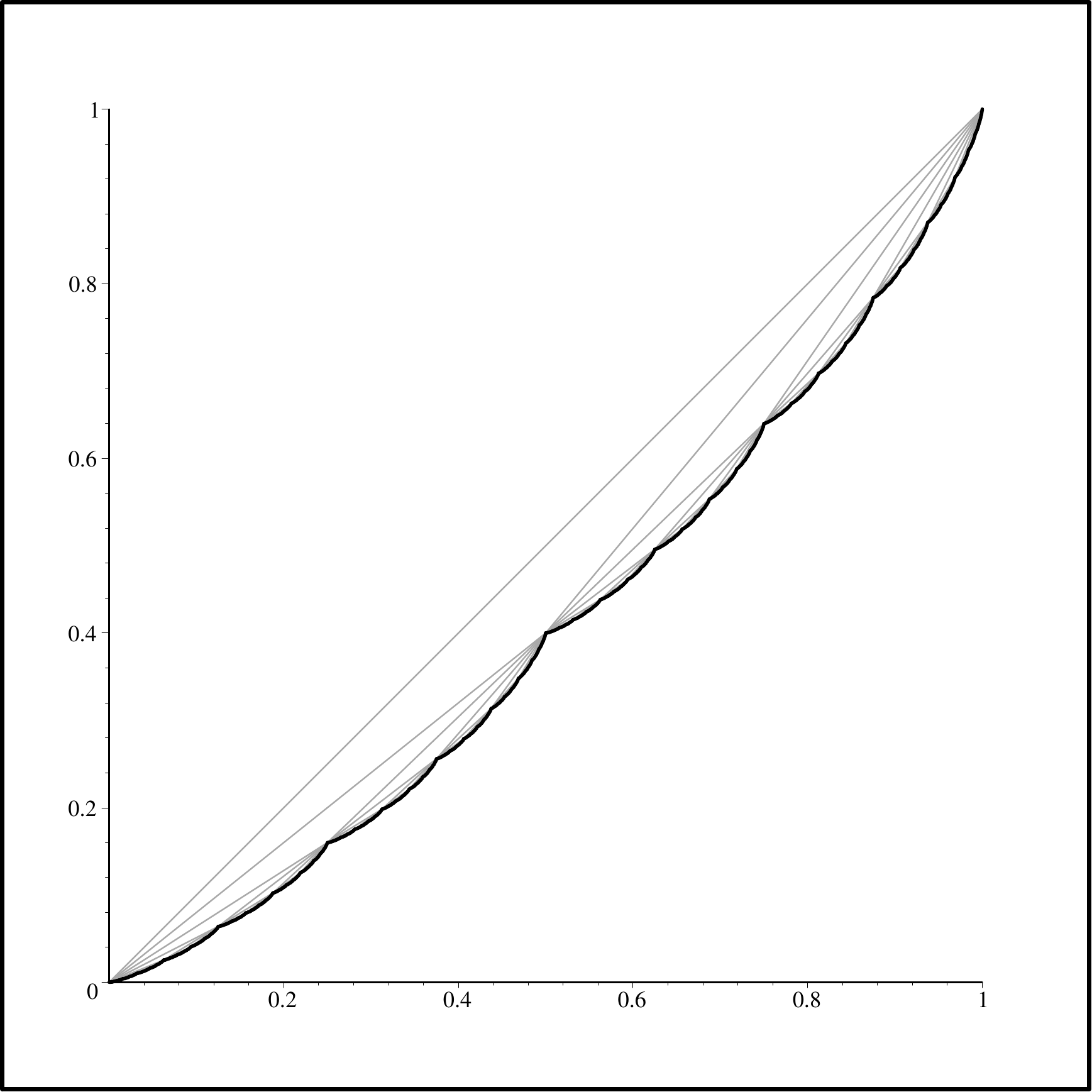}  
  \end{center}
\caption{\small The line segment from $(0,0)$ to $(1,1)$ after applying $\mathcal{S}_{s}$, Salem's iterated function system, $0$, $4$ and $10$ times, respectively.}
\label{fig:Salem}
\end{figure}

\noindent In the following proposition, we show that we can interpret the attractor $\mathcal{A}_s$ as the restriction to $[0,1]$ of the graph of the function $F_s:\mathbb{R}\to\mathbb{R}$ satisfying \begin{equation}\label{eq:Fsalem}F_s(x)=\frac{2}{5}\cdot F_s(2x)+\frac{3}{5}\cdot F_s(2x-1),\quad\mbox{where}\quad F_s(x)=\begin{cases} 0 & x\leqslant 0\\ 1 & x\geqslant 1\end{cases}. \end{equation}

\begin{proposition}\label{salem} Let $F_s$ be the unique solution to the dilation equation \eqref{eq:Fsalem}. The set $\big\{(x,F_s(x)):x\in[0,1]\big\}$ is the attractor $\mathcal{A}_s$ of Salem's iterated function system $\mathcal{S}_s$. 
\end{proposition}

\begin{proof} We write $\mathcal{F}_s:=\{(x,F_s(x)):x\in[0,1]\big\}$, and want to show that $\mathcal{S}(\mathcal{F}_s)=\mathcal{F}_s$. As a first step, we note that $(0,0)\in\mathcal{F}_s$ so that this set is not trivially empty. Second, we take $(x)_2=.x_1x_2\cdots\in[0,1]$ and consider $(x,F_s(x))$. Now, \begin{multline*}\qquad S_0\left(\begin{matrix} x\\ F_s(x)\end{matrix}\right)=\left(\begin{matrix} 1/2 & 0\\ 0 &2/5\end{matrix}\right)\left(\begin{matrix} x\\ F_s(x)\end{matrix}\right)\\ =\left(\begin{matrix} x/2\\ (2/5)F_s(x)\end{matrix}\right)=\left(\begin{matrix} .0x_1x_2\cdots\\ (2/5)F_s(x)\end{matrix}\right)=\left(\begin{matrix} y\\ (2/5)F_s(2y)\end{matrix}\right),\qquad \end{multline*} where $y:=.0x_1x_2\cdots$, so $x=2y$. Note that $y\in[0,1/2]$. Thus $2y-1\leqslant 0$, so \eqref{eq:Fsalem} gives that $$\frac{2}{5}\cdot F_s(2y)=\frac{2}{5}\cdot F_s(2y)+\frac{3}{5}\cdot F_s(2y-1)=F_s(y),$$ and so $S_0(\mathcal{F}_s)\subseteq \mathcal{F}_s$. Similarly, \begin{multline*} \qquad S_1\left(\begin{matrix} x\\ F_s(x)\end{matrix}\right)=\left(\begin{matrix} 1/2 & 0\\ 0 &3/5\end{matrix}\right)\left(\begin{matrix} x\\ F_s(x)\end{matrix}\right)+\left(\begin{matrix} 1/2 \\ 2/5\end{matrix}\right)\\ =\left(\begin{matrix} x/2+1/2\\ (3/5)F_s(x)+2/5\end{matrix}\right)=\left(\begin{matrix} .1x_1x_2\cdots\\ (3/5)F_s(x)+2/5\end{matrix}\right)\\ =\left(\begin{matrix} w\\ (3/5)F_s(2w-1)+2/5\end{matrix}\right),\qquad \end{multline*} where $w:=.1x_1x_2\cdots$, so $x=2w-1$. Note that $w\in[1/2,1]$. Thus $2w\geqslant 1$, so \eqref{eq:Fsalem} gives that $$\frac{2}{5}+\frac{3}{5}\cdot F_s(2w-1)=\frac{2}{5}\cdot F_s(2w)+\frac{3}{5}\cdot F_s(2w-1)=F_s(w),$$ and so $S_1(\mathcal{F}_s)\subseteq \mathcal{F}_s$, thus $\mathcal{S}_s(\mathcal{F}_s)\subseteq \mathcal{F}_s$. To complete the proof, we need only note that the reverse inclusion follows by using the above equalities backwards.
\end{proof}

The proof of the general result follows {\em mutatis mutandis}.

\begin{proof}[Proof of Theorem \ref{thm:dilationtoifs}] Our proof follows the method of the special case recorded in Proposition~\ref{salem} above. To this end, let $(x)_k=.x_1x_2\cdots\in[0,1]$ and consider $(x,{\bf F}(x))$. For $j\in\{0,1\ldots,k-1\}$, we have \begin{align*}S_j\left(\begin{matrix}x\\ {\bf F}(x)\end{matrix}\right) &=\left(\begin{matrix} 1/k & {\bf 0}^{1\times (k-1)}\\ {\bf 0}^{(k-1)\times 1} & \rho^{-1}\, {\bf B}_j\end{matrix}\right)\left(\begin{matrix}x\\ {\bf F}(x)\end{matrix}\right)+\left(\begin{matrix}j/k\\  \sum_{a=0}^{j-1}\rho^{-1}\, {\bf B}_a{\bf v}_\rho\end{matrix}\right)\\
&=\left(\begin{matrix} x/k+j/k\\ \rho^{-1}\, {\bf B}_j{\bf F}(x)+\sum_{a=0}^{j-1}\rho^{-1}\, {\bf B}_a{\bf v}_\rho\end{matrix}\right)\\
&=\left(\begin{matrix} .jx_1x_2\cdots\\ \rho^{-1}\, {\bf B}_j{\bf F}(k(.jx_1x_2\cdots)-j)+\sum_{a=0}^{j-1}\rho^{-1}\, {\bf B}_a{\bf v}_\rho\end{matrix}\right).
\end{align*} To finish the proof, it is enough to show that \begin{equation}\label{fjxV}{\bf F}(.jx_1x_2\cdots)= \rho^{-1}\, {\bf B}_j{\bf F}(k(.jx_1x_2\cdots)-j)+\sum_{a=0}^{j-1}\rho^{-1}\, {\bf B}_a{\bf v}_\rho.\end{equation} We compute ${\bf F}(.jx_1x_2\cdots)$ using \eqref{dilation}. Since $y:=.jx_1x_2\cdots\in[0,1]$, we have that $ky-a\geqslant 1$ for $a\in\{0,\ldots,j-1\}$, and $ky-a\leqslant 0$ for $a\in\{j+1,\ldots,k-1\}$. Thus, for this $y$, using the definition of ${\bf F}(x)$ outside of $(0,1)$, we have ${\bf F}(ky-a)={\bf F}(1)={\bf v}_\rho,$ for $a\in\{0,\ldots,j-1\}$, and ${\bf F}(ky-a)={\bf 0},$ for $a\in\{j+1,\ldots,k-1\}$. Thus \begin{align*} {\bf F}(y)&=\sum_{a=0}^{k-1}\rho^{-1}\, {\bf B}_{a} {\bf F}(k y-a)\\
&=\sum_{a=j+1}^{k-1}\rho^{-1}\, {\bf B}_{a} {\bf F}(k y-a)+\rho^{-1}\, {\bf B}_{j} {\bf F}(k y-j)+\sum_{a=0}^{j-1}\rho^{-1}\, {\bf B}_{a}  {\bf F}(k y-a)\\
&={\bf 0}+\rho^{-1}\, {\bf B}_{j} {\bf F}(k y-j)+\sum_{a=0}^{j-1}\rho^{-1}\, {\bf B}_{a}  {\bf F}(1)\\ &=\rho^{-1}\, {\bf B}_{j} {\bf F}(k y-j)+\sum_{a=0}^{j-1}\rho^{-1}\, {\bf B}_{a}  {\bf v}_\rho,
\end{align*} which shows that $\mathcal{S}_f(\mathcal{F}_f)\subseteq \mathcal{F}_f$. Again, as in the proof of Proposition~\ref{salem}, the reverse inclusion follows using the above equalities in the reverse direction.
\end{proof}

In our context, one can make an interesting addition to the above attractor-dilation equation relationship---a regular sequence associated to the digit matrices. We can do this with the above example of Salem \cite{S1943} as follows.

\begin{definition} Let $k\geqslant 2$ be a positive integer. We call a $1$-dimensional nonnegative regular sequence $f$ a {\em Salem sequence} and refer to the digit matrices of a Salem sequence simply as {\em digits}.
\end{definition}

\noindent For example, the Salem attractor $\mathcal{A}_s$ is related to the $2$-regular Salem sequence $s=(s(n))_{n\geqslant 0}$ given by $s(n)=2^{s_0(n)}3^{s_1(n)},$ where $s_0(n)$ and $s_1(n)$ counts the number of zeros and ones, respectively, in the binary expansion of $n$---the careful reader will notice the use of the subscript $s$ in the above discussion of Salem's iterated function system. Here, the ghost distribution of $s$ is exactly the solution, $F_s$, of the dilation equation \eqref{eq:Fsalem}. In general, the relationship between the ghost distribution and the solution of the dilation equation is slightly more complicated. A recent result, recorded below, makes this relationship between the regular sequence, the dilation equation and the attractor explicit; see Theorem 5 and its proof in Coons, Evans and Ma\~nibo~\cite{CEM}.

\begin{theorem}[Coons, Evans and Ma\~nibo]\label{thm:cem} Let $k\geqslant 2$ be an integer and $f$ be a $k$-regular sequence. Suppose that the spectral radius $\rho({\bf B})$ is the unique dominant eigenvalue of ${\bf B}$, that $$\rho:=\rho({\bf B})>\rho^*(\{{\bf B}_0,\ldots,{\bf B}_{k-1}\})=:\rho^*,$$ that for $n$ large enough $\varSigma(n)\neq 0$ and that the asymptotical behaviour of $\varSigma(n)$ is determined by $\rho({\bf B})$. Then 
\begin{equation}\label{eq:mufaffine} 
\mu_{f}([0,x])=\frac{e_1^T \left({\bf F}\big(\frac{1+(k-1)x}{k}\big)-{\bf F}\big(\frac{1}{k}\big)\right)}{e_1^T \left({\bf F}(1)-{\bf F}\big(\frac{1}{k}\big)\right)}.\end{equation}
\end{theorem}

\begin{remark}
Recall from the definition of a $k$-regular sequence that the $\mathbb{Q}$-vector space $\mathcal{V}_{k}(f)$ generated by the $k$-kernel is generated by the set $\left\{f_1,\ldots,f_d\right\}$, where we have set $f_1:=f$. From the solution ${\bf F}(x)$, one can in fact recover the ghost distributions for all $f_i$ by looking at the $i$th coordinate of ${\bf F}(x)$, which can be achieved by replacing $e^{T}_1$ with $e^{T}_i$ in Eq.~\eqref{eq:mufaffine}.
In the special case when $f$ is primitive, it follows from \cite[Thm.~1]{CEM} that the ghost measures associated to each $f_i$ are the same, and hence it suffices to consider $f_1$ in this setting.
\end{remark}

Combining Theorem \ref{thm:cem} with Theorem \ref{thm:dilationtoifs}, we have the following result, which immediately implies Theorem \ref{thm:ghostselfaffine}.

\begin{theorem}\label{main} Let $k\geqslant 2$ be an integer and $f$ be a $k$-regular sequence. Suppose that the spectral radius $\rho({\bf B})$ is the unique dominant eigenvalue of ${\bf B}$, that $$\rho:=\rho({\bf B})>\rho^*(\{{\bf B}_0,\ldots,{\bf B}_{k-1}\})=:\rho^*,$$ that for $n$ large enough $\varSigma(n)\neq 0$ and that the asymptotical behaviour of $\varSigma(n)$ is determined by $\rho({\bf B})$. Then the ghost distribution of $f$ is an affine transformation of a self-affine set.
\end{theorem}

\begin{proof} We have two cases depending on the distinctness of the digit matrices of $f$. Firstly, if all of the digit matrices are equal, then the sequence $f$ is constant between powers of $k$, and so the ghost measure is just standard Lebesgue measure, whose distribution function is clearly self-affine. Secondly, suppose the digit matrices of $f$ are not all equal. Combining the result of Theorem \ref{thm:dilationtoifs}---that the solution of the related dilation equation is self-affine---we obtain the desired conclusion noting that the relationship in \eqref{eq:mufaffine} gives that the ghost distribution is the affine image of a solution of a dilation equation. 
\end{proof}

\section{Characterisation of $1$-dimensional ghost distributions}\label{sec:1d}

In this section, we characterise the spectral type of the ghost measures of $1$-dimensional regular sequences, that is, the Salem sequences. 

It is clear that if the digits (nonnegative integers\footnote{In this paper, we have assumed that regular sequences are integer-valued. In general, this is not necessary. In particular, Proposition \ref{prop:salemseq} holds for Salem sequences with nonnegative real digits.}) of a nonzero Salem sequence are all equal then the resulting ghost measure is precisely Lebesgue measure. The result is quite different when the digits are not all equal. Indeed, Salem \cite{S1943} showed that the ghost distributions of positive $2$-regular Salem sequences whose two digits are not equal are singular continuous. The main result of this section is the following proposition.

\begin{proposition}\label{prop:salemseq} Let $f$ be a $k$-regular Salem sequence with digits $b_0,b_1,\ldots,$ $b_{k-1}$ that are not all equal.  Then 
\begin{enumerate}
\item[(a)] If only $b_0$ is nonzero, then the ghost measure is the zero measure.
\item[(b)] If only one of the digits is nonzero, and $b_0=0$, then the ghost measure is pure point.
\item[(c)] If more than one of the digits is nonzero, then the ghost measure is singular continuous.
\end{enumerate}
\end{proposition}

\begin{proof} Parts (a) and (b) are straightforward. If $b_0$ is the only nonzero digit, then since all base expansions of positive integers start with a digit other than zero, the Salem sequence is the zero sequence, and so its measure is the zero measure; this proves (a). If only one of the digits is nonzero, say $b_j$ with $j\neq 0$, then the ghost measure is supported on the single number in $[0,1]$ whose $k$-ary expansion is $0.jjjjj\cdots$, so that the ghost measure is equal to the delta measure $\delta_{j/(k-1)}$; this proves (b).

Case (c) breaks into two cases. The first is also straightforward---the case that one of the digits is zero. For if one of the digits is zero and at least two are not zero, then the ghost measure is supported on an uncountable set of measure zero: the set of numbers in $[0,1]$ whose $k$-ary expansions contain only the $k$-ary digits $j$ with $b_j\neq 0$. 

Thus we can suppose that none of the $b_j$ are zero. Using Theorem \ref{thm:ghostselfaffine}, to show that the ghost measure is singular continuous it is enough to show that the associated Salem attractor, or solution to the related dilation equation, provided for in Theorem~\ref{thm:dilationtoifs} is singular continuous; that is, the attractor of the iterated function system $\mathcal{S}:=\{S_0,\ldots,S_{k-1}\}$ with $$S_j\left(\begin{matrix}x\\ y\end{matrix}\right)=\left(\begin{matrix} 1/k & 0\\ 0 & b_j/b\end{matrix}\right)\left(\begin{matrix}x\\ y\end{matrix}\right)+\left(\begin{matrix}j/k\\  \sum_{a=0}^{j-1}b_a/b\end{matrix}\right),$$ where $b:=b_0+b_1+\cdots+b_{k-1}$. We write the points of the attractor as $(x,S(x))$, and note that $S(x)$ is the solution to the dilation equation $$S(x)=\frac{1}{b}\, \sum_{a=0}^{k-1}\, b_a\cdot S(kx-a),\quad\mbox{where}\quad S(x)=\begin{cases} 0 & x\leqslant 0\\ 1 & x\geqslant 1\end{cases}.$$ See Figure \ref{fig:Salem} for a plot of both the Salem attractor and ghost distribution for the case $k=2$ and $(b_0,b_1)=(2,3)$; the grey boxes illustrate Theorem \ref{thm:cem}.

We use an old argument of Salem involving simply normal numbers. Recall that a real number $x$ is simply normal to the base $k$, provided each $k$-ary digit occurs in its base expansion with frequency $1/k$, and that  Lebesgue-almost all numbers in $[0,1]$ are simply normal. Let $x\in[0,1]$ be a simply normal number with base-$k$ expansion $(x)_k=0.x_1x_2x_3\cdots$. Note that for such $x$ and any $j\in\{0,1,\ldots,k-1\}$, we have that the number of $x_i=j$ with $i$ up to $n$ is $n/k+o(n)$. Now, for each $n\geqslant 0$, set $$y_n=x+\frac{b_{n+1}}{k^{n+1}},\quad\mbox{where}\quad b_{n+1}=\begin{cases}1 &\mbox{if $x_{n+1}=0$}\\ -1 &\mbox{otherwise}.\end{cases}$$ The first $n$ digits in the $k$-ary expansion of $y_n$ agrees with those of $x$, thus 
\begin{align*}
\left|\frac{S(x)-S(y_n)}{x-y_n}\right|
&<k^{n+1}(b_{x_1}/b)(b_{x_2}/b)\cdots(b_{x_n}/b)\\
&\leqslant \left(\frac{k}{b}\prod_{j=0}^{k-1}b_{j}^{1/k}\right)^{n}\cdot  k\cdot\left(\prod_{j=0}^{k-1}b_{j}\right)^{|r(n)|},
\end{align*} 
where $|r(n)|=o(n)$ as $n\to\infty$. Also, since $S$ is continuous and increasing, the derivative of $S$ exists almost everywhere. To finish our proof, it is enough to invoke the arithmetic-geometric mean inequality, to see that \begin{equation*}\frac{k}{b}\prod_{j=0}^{k-1}b_{j}^{1/k}<\frac{k}{b}\left(\frac{b_0+b_1+\cdots+b_{k-1}}{k}\right)=1,\end{equation*} so that $|S(x)-S(y_n)|/|x-y_n|\to0$ as $n\to\infty$.
\end{proof}

\begin{figure}[ht]
\begin{center}
  \includegraphics[width=.98\linewidth]{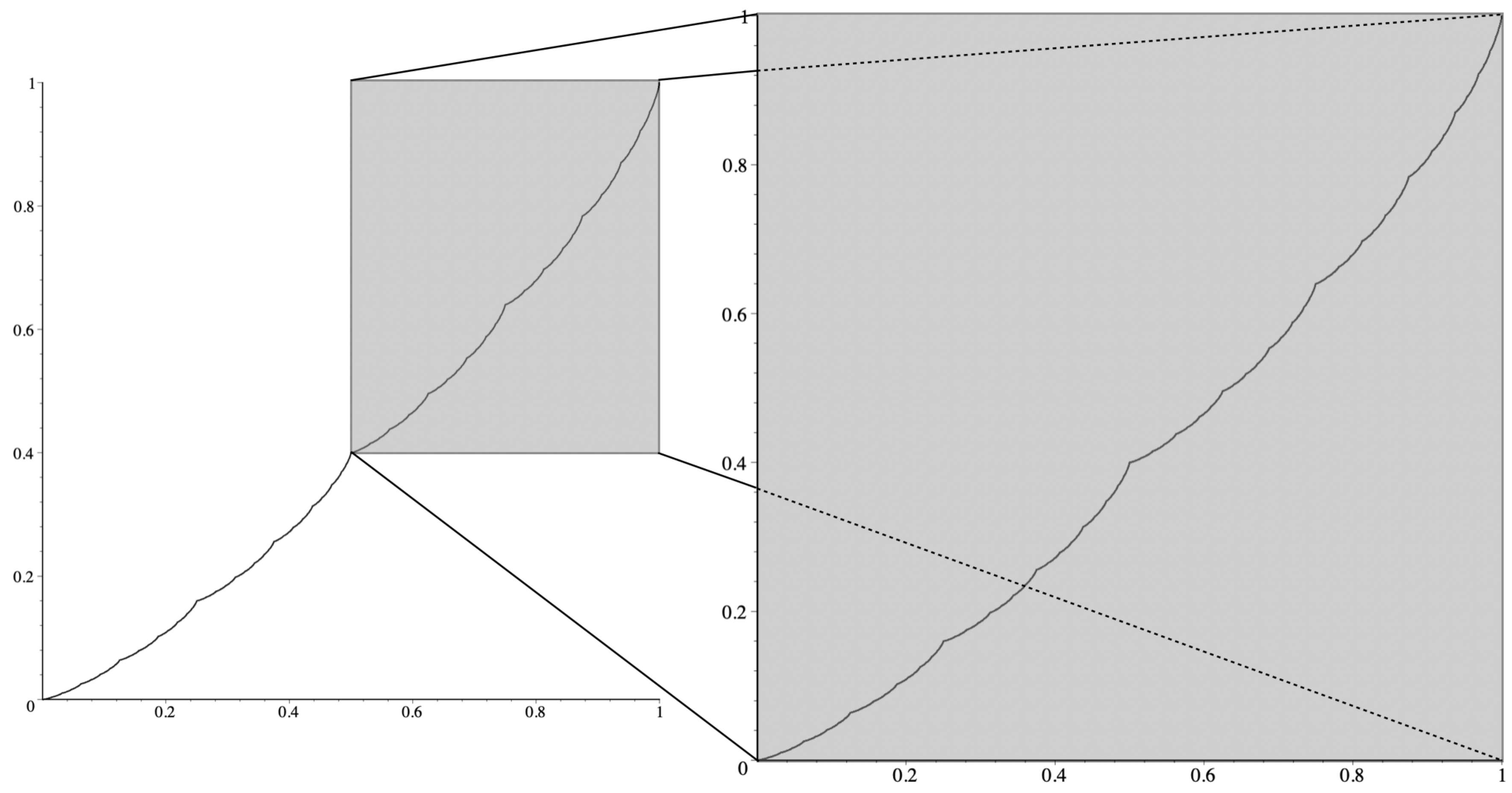}
  \end{center}
\caption{\small The $2$-regular Salem ghost distribution with $(b_0,b_1)=(2,3)$ (right) is an affine image of a section of the related Salem attractor (left).}
\label{fig:Salem}
\end{figure}

The careful reader acquainted with so-called `missing-digit' sets, such as the middle-third Cantor set, may recognise that the ghost measures of Salem sequences are precisely the generalisation of standard mass distributions supported on missing-digit sets. For example, the standard Cantor measure, supported on real numbers $x\in[0,1]$ with a ternary expansion not containing the digit $1$, is the ghost measure of the $3$-regular Salem sequence with $(b_0,b_1,b_2)=(1,0,1)$. Unlike with standard missing-digit distributions, corresponding to choosing only the digits zero or one for our Salem sequence construction, in our context we can choose different weights, which then, by Proposition \ref{prop:salemseq}, give new singular continuous distributions; see Figure~\ref{fig:Cants} for the standard Cantor distribution (the Devil's staircase) and two Salem sequence variants.

\begin{figure}[ht]
\begin{center}
  \includegraphics[width=.32\linewidth]{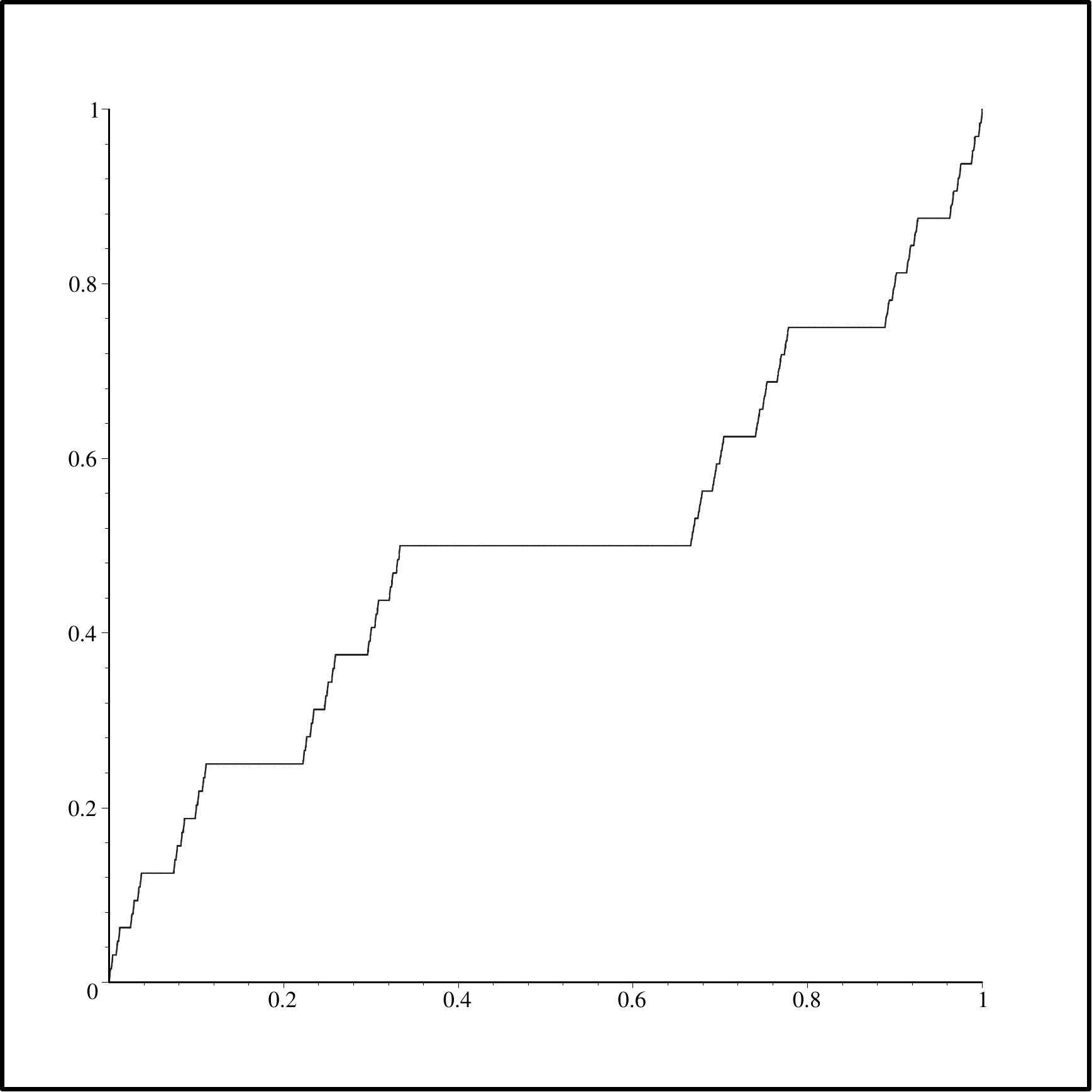}
  \includegraphics[width=.32\linewidth]{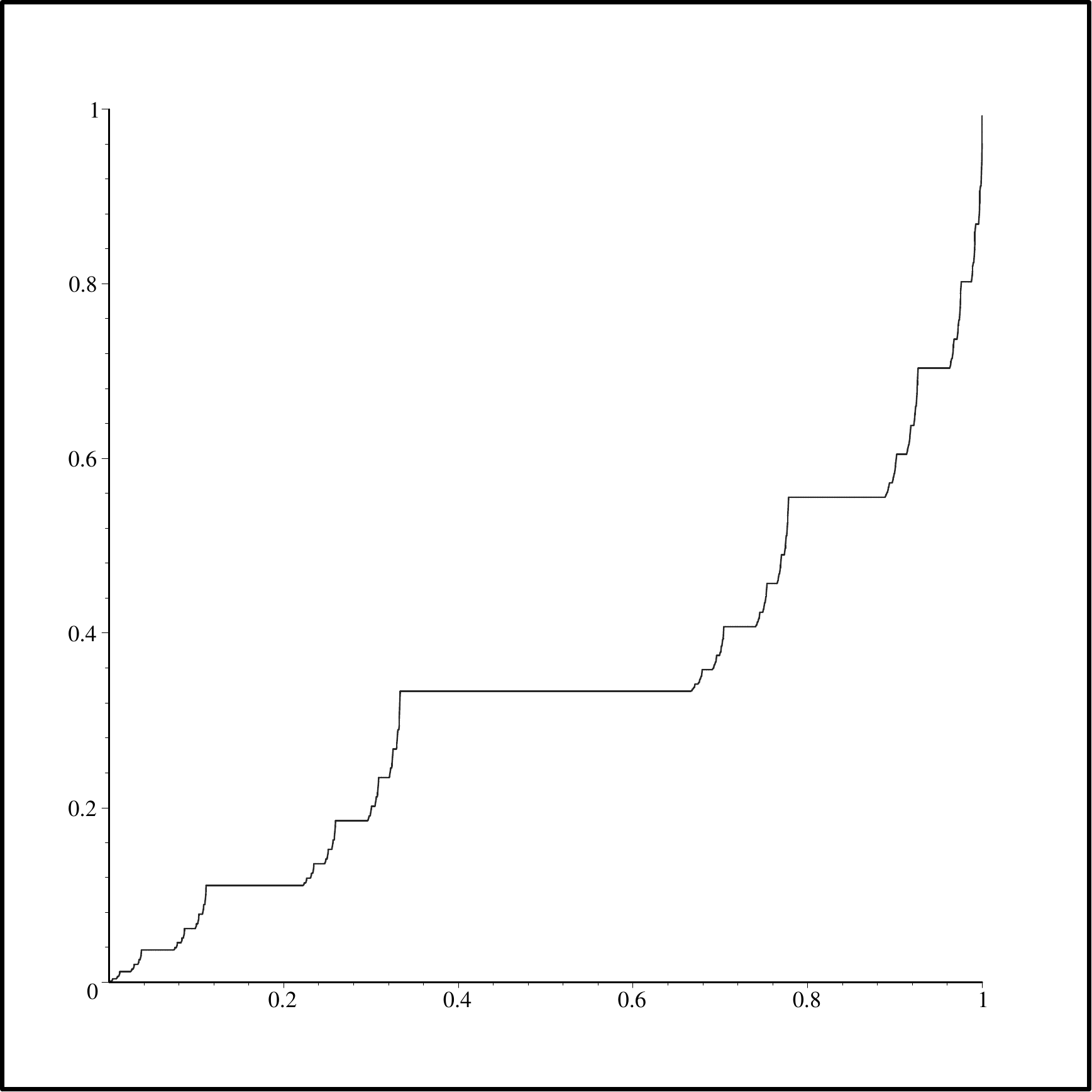} 
  \includegraphics[width=.32\linewidth]{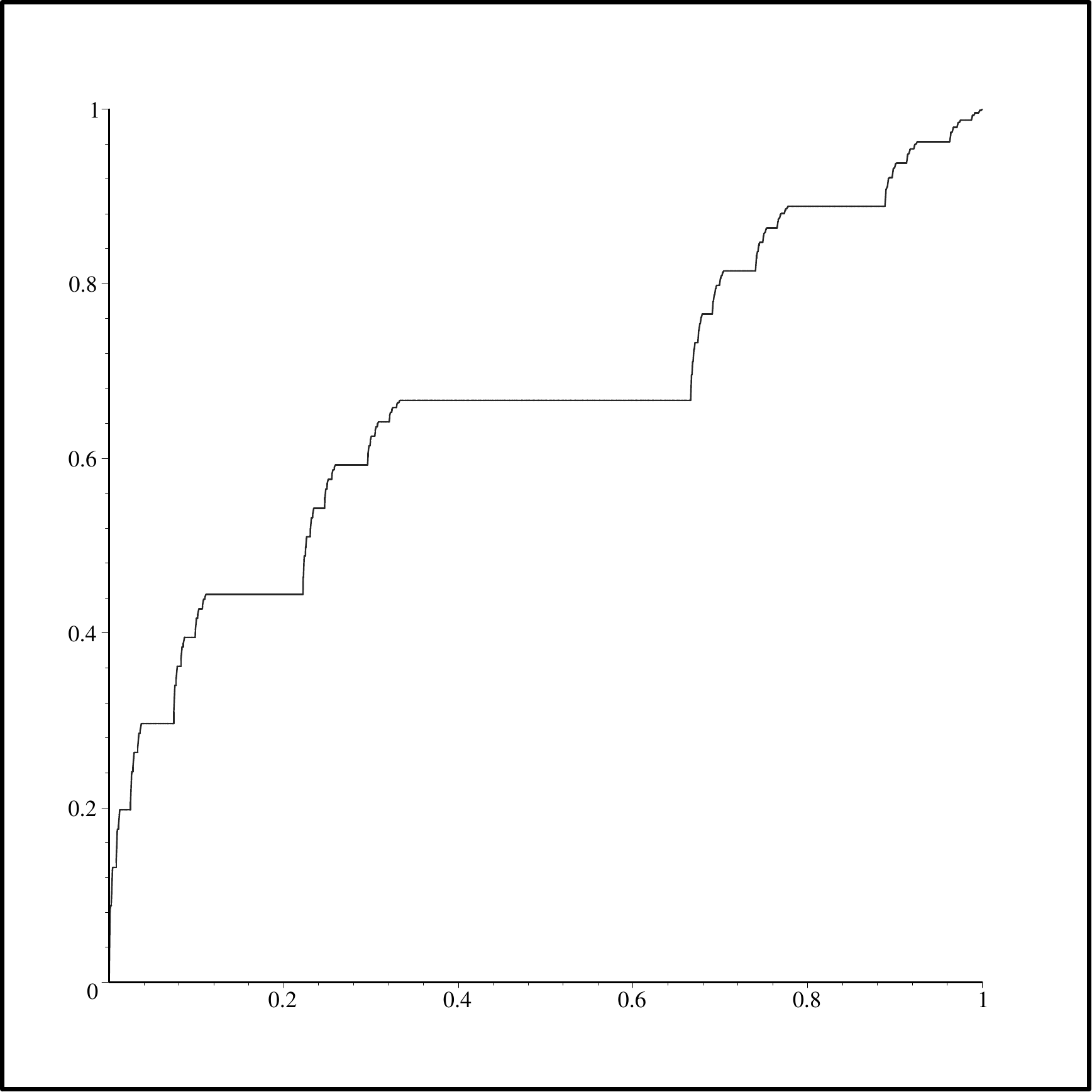}  
  \end{center}
\caption{\small The Devil's staircase (left)---the ghost distribution of the standard Cantor middle-thirds set---along with the ghost distributions of the $3$--regular Salem sequences with digits $(b_0,b_1,b_2)$ equal to $(1,0,2)$ (middle) and $(2,0,1)$ (right).}
\label{fig:Cants}
\end{figure}

\begin{remark} 
It is interesting to note that the Cantor distribution was one of the first examples of a singular continuous function, while Salem's contribution was ``to give simple, direct constructions of strictly increasing singular functions.'' While at the time these seemed to be very different ideas, our definition of the Salem sequence shows in fact that these are just two examples of the same phenomenon.
\end{remark}

\section{Singular continuity of Zaremba's ghost distributions}\label{sec:zaremba}

For all $x\in(0,1)$ we write the ordinary continued fraction expansion of $x$ as $x=[a_1,a_2,a_3,\ldots],$ where the positive integers $a_1,a_2,a_3,\ldots,$ are the partial quotients of $x$. The convergents of the number $x$ are the rationals $p_n/q_n:=[a_1,\ldots,a_n]$ and can be computed using the well-known relationship 
$$\left(\begin{matrix} a_n&1\\ 1&0\end{matrix}\right)\cdots\left(\begin{matrix} a_2&1\\ 1&0\end{matrix}\right)\left(\begin{matrix} a_1&1\\ 1&0\end{matrix}\right)=\left(\begin{matrix} q_n&p_{n}\\ q_{n-1}&p_{n-1}\end{matrix}\right).$$ 
See the monograph ``Neverending Fractions'' by Borwein, van der Poorten, Shallit and Zudilin \cite{BvdPSZ} for details and properties regarding continued fractions. Denote by $\mathfrak{D}_k$ the set of denominators of convergents of real numbers $x\in(0,1)$ all of whose partial quotients are bounded above by $k$. Zaremba \cite{Z1972} conjectured that {\em there is a positive integer $k$ such that $\mathfrak{D}_k=\mathbb{N}$.} Bourgain and Kontorovich \cite{BK2014} proved that the set of denominators $\mathfrak{D}_{50}$ has full density in $\mathbb{N}$; this was improved by Huang \cite{H2015} and Jenkinson and Pollicott \cite{JP2019}, who proved the analogous result for $\mathfrak{D}_5$. 

In this section, we study ghost measures of the so-called {\em Zaremba sequences}; see Coons \cite{C2018} for more details on these sequences and results relating to their power series generating functions. For the purposes of this section, we call the $k$-regular sequence $z_k$ a {\em $k$-Zaremba sequence} provided it has linear representation \eqref{eq:linear} satisfying ${\bf w}=(1\ 0)$ and $${\bf B}_i=\left(\begin{matrix} i+1 & 1\\ 1& 0\end{matrix}\right),\quad\mbox{for } i=0,1,\ldots,k-1.$$ Note that the sequence $z_k$ is an ordering of the set $\mathfrak{D}_k$. In this way, Zaremba's conjecture can be stated as a spectral question about $z_k$; in particular, {\em is there an integer $k$ such that $z_k$ takes all integer values?} Of course, in this paper we are not focusing on Zaremba's conjecture, but on the nature\footnote{Of course, it would be extremely interesting if one could pull out a result on the values of $z_k$ by examining its ghost distribution... It does seem likely that if the ghost distribution of a regular sequence is not continuous, then the sequence would not take all integer values.} of the related ghost measure and ghost distribution. See Figure \ref{fig:Zaremba} for a plot of the $2$-Zaremba ghost distribution and its related attractor.

\begin{figure}[ht]
\begin{center}
  \includegraphics[width=.98\linewidth]{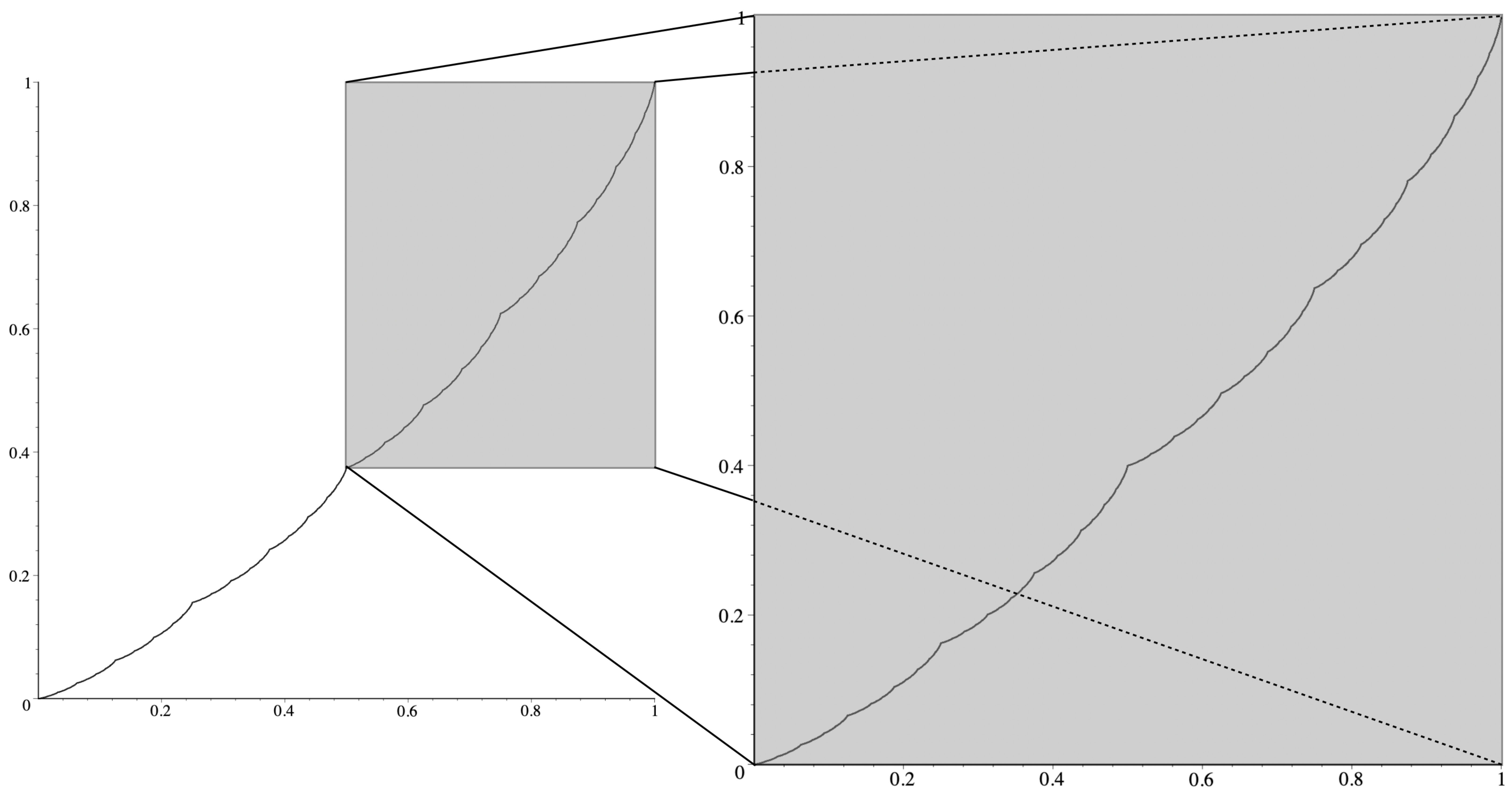}
  \end{center}
\caption{The $2$-Zaremba ghost distribution (right) is an affine image of a section of the related $2$-Zaremba attractor (left).}
\label{fig:Zaremba}
\end{figure}

A result of Coons, Evans and Ma\~nibo \cite[Thm.~6]{CEM} applies to give that the ghost measure of $z_k$ is continuous. The main result of this section is to show that this distribution is singular.

\begin{theorem}\label{thm:zsc} The ghost measure of the Zaremba sequence $z_k$ is singular continuous.
\end{theorem}

Theorem \ref{thm:zsc} is an immediate consequence of the following more general result.

\begin{theorem}\label{thm:genmain} Let $k\geqslant 2$ be an integer and $f$ be a nonnegative $k$-regular sequence. Suppose that the spectral radius $\rho({\bf B})$ is the unique dominant eigenvalue of ${\bf B}$, that $$\rho=\rho({\bf B})>\rho^*(\{{\bf B}_0,\ldots,{\bf B}_{k-1}\})=\rho^*,$$ that for $n$ large enough $\varSigma(n)\neq 0$ and that the asymptotical behaviour of $\varSigma(n)$ is determined by $\rho({\bf B})$. If \begin{equation}\label{eq:agm}\prod_{j=0}^{k-1}\|{\bf B}_{j}\|^{1/k}<\frac{\rho}{k},\end{equation} then the ghost measure of $f$ is purely singular continuous.
\end{theorem}

\begin{proof} Since the ghost distribution of such a sequence is an affine transformation of a related attractor, it is enough to prove that this attractor is singular continuous. Note also that since $f$ is nonnegative, its ghost distribution and attractor are both increasing functions on $[0,1]$, and so of bounded variation. Thus they are almost everywhere differentiable. 

Denote the points of the curve defined by the first two coordinates of the solution of the related dilation equation or attractor by $(x,S(x))$. As in the proof of Theorem~\ref{prop:salemseq} above, we let $x\in[0,1]$ be a simply normal number with base-$k$ expansion $$(x)_k=0.x_1x_2x_3\cdots.$$ We remind the reader that for such $x$ and any $j\in\{0,1,\ldots,k-1\}$, we have that the number of $x_i=j$ with $i$ up to $n$ is $n/k+o(n)$. Now, for each $n\geqslant 0$, set $$y_n=x+\frac{b_{n+1}}{k^{n+1}},\quad\mbox{where}\quad b_{n+1}=\begin{cases}1 &\mbox{if $x_{n+1}=0$}\\ -1 &\mbox{otherwise}.\end{cases}$$ The first $n$ bits in the binary expansion of $y_n$ agree with those of $x$, thus 
\begin{align*}
\left|\frac{S(x)-S(y_n)}{x-y_n}\right|
&<k^{n+1}\cdot|e_1^T(\rho^{-1}{\bf B}_{x_1})(\rho^{-1}{\bf B}_{x_2})\cdots(\rho^{-1}{\bf B}_{x_n}){\bf v}_\rho|\\
&<(k\rho^{-1})^{n}\cdot\|{\bf B}_{x_1}{\bf B}_{x_2}\cdots{\bf B}_{x_n}\|\cdot k\cdot \|e_1^T\|\cdot \|{\bf v}_\rho\| \\
&\leqslant \left(k\rho^{-1}\prod_{j=0}^{k-1}\|{\bf B}_{j}\|^{1/k}\right)^{n}\cdot  k\cdot \|{\bf v}_\rho\|\left(\prod_{j=0}^{k-1}\|{\bf B}_{j}\|\right)^{|r(n)|},
\end{align*} 
where $|r(n)|=o(n)$ as $n\to\infty$, and, as it is throughout this paper, $\|\cdot\|$ is the operator norm for a matrix and the standard Euclidean norm for a vector. But since we have assumed that \eqref{eq:agm} holds, there is a constant $c<1$ such that $$\left|\frac{S(x)-S(y_n)}{x-y_n}\right|<c^n\cdot k\cdot \|{\bf v}\|\left(\prod_{j=0}^{k-1}\|{\bf B}_{j}\|\right)^{|r(n)|}=o(1).$$ Since the derivative of $S(x)$ exists almost everywhere, by the above, it is is zero for Lebesgue-almost all $x\in[0,1]$; that is, $S(x)$, and therefore $\mu_{z_k}([0,x])$, is singular continuous.
\end{proof}

Recall that $$\rho=\rho({\bf B})=\rho({\bf B}_0+{\bf B}_1+\cdots+{\bf B}_{k-1})\leqslant \|{\bf B}_0+{\bf B}_1+\cdots+{\bf B}_{k-1}\|.$$ Thus, using submultiplicativity of the operator norm for matrices, if \eqref{eq:agm} holds, it implies that $$\|{\bf B}_{0}{\bf B}_{1}\cdots{\bf B}_{k-1}\|^{1/k}< \frac{1}{k}\cdot\|{\bf B}_0+{\bf B}_1+\cdots+{\bf B}_{k-1}\|,$$ which is exactly the arithmetic-geometric mean inequality for the operator norm of matrices. That is, \eqref{eq:agm} is a strong version of the arithmetic-geometric mean inequality (recall that our assumptions imply that the ${\bf B}_i$ are not all equal). As we shall see, in the case of the Zaremba sequence $z_k$ such an inequality holds.

\begin{proof}[Proof of Theorem \ref{thm:zsc}] Since the matrices $${\bf B}_i=\left(\begin{matrix} i+1 & 1\\ 1& 0\end{matrix}\right),$$ for the Zaremba sequence are all symmetric and nonnegative, we have that the operator norm is equal to the largest eigenvalue. Thus $$\|{\bf B}_i\|=\frac{i+1+\sqrt{(i+1)^2+4}}{2}\qquad\mbox{and}\qquad \rho=\frac{k(k+1)+k\sqrt{(k+1)^2+16}}{4}.$$ Since $k\geqslant 2$, we have $$\prod_{i=0}^{k-1}\|{\bf B}_{i}\|^{1/k}=\frac{1}{2}\prod_{i=1}^k\left(i+\sqrt{i^2+4}\right)^{1/k}<\frac{(k+1)+\sqrt{(k+1)^2+16}}{4}=\frac{\rho}{k}.$$ An application of Theorem \ref{thm:genmain} finishes the proof.
\end{proof}

Note that \eqref{eq:agm} does not hold in general. It is not even true for strictly positive symmetric matrices. For a counterexample, note that 
$$\left\|\left(\begin{matrix} 1&1\\ 1&1\end{matrix}\right)\right\|^{1/2}\left\|\left(\begin{matrix} 2&1\\ 1&1\end{matrix}\right)\right\|^{1/2}\hspace{-.1cm}=\left(3+\sqrt{5}\right)^{1/2}\hspace{-.1cm}>2.288>\frac{5+\sqrt{17}}{4}=\frac{1}{2}\cdot\rho\left(\left(\begin{matrix} 3&2\\ 2&2\end{matrix}\right)\right)\hspace{-.1cm}.$$

Singularity criteria which depend on scaling inequalities like in Theorem~\ref{thm:genmain} are prevalent in spectral theory. As an example, for a primitive substitution with inflation factor $\theta$, if the associated Lyapunov exponent $\chi$ is strictly less than $\log\sqrt{\theta}$, the spectrum of the system is purely singular. 

Note that from the matrix semigroup $\langle\mathcal{B}\rangle$, under some natural assumptions on the matrices ${\bf B}_i$, one can associate to it a Lyapunov exponent $\chi^{ }_{\mathcal{B}}$  given by
\[
\chi^{ }_{\mathcal{B}}=\lim_{n\to\infty}\frac{1}{n}\log\|{\bf B}_{i_{n-1}}\cdots {\bf B}_{i_0}\|,
\]
which, by the seminal result of Furstenberg and Kesten \cite{FK1960}, exists and is constant for almost every sequence $(i_n)_{n\geqslant 0}\in \Sigma=\left\{0,\ldots,k-1\right\}^{\mathbb{N}}$. Trivially, $\ee^{\chi^{ }_{\mathcal{B}}}\leqslant \rho^{\ast}$. Another sufficient condition for the singularity of the ghost measure is that $\chi^{ }_{\mathcal{B}}<\log(\rho/k)$. We note, however, that Lyapunov exponents are difficult to compute in general---the conditions in Theorem~\ref{thm:genmain} are comparatively easier to verify for concrete examples.

\section{Concluding remarks}\label{sec:conclusion}

In this paper, we connected the ghost measure of a regular sequence to an attractor of an iterated function system of affine contractions. We then proved that certain ghost measures are singular continuous by showing that the related attractor is a singular continuous curve. Our results rested on the set of given digit matrices satisfying a strong version of the arithmetic-geometric mean inequality; see inequality \eqref{eq:agm}. This is used to show that, in the fundamental region $\bigl[k^m,k^{m+1}\bigr)$, most values of a regular sequence are smaller than the average value of the sequence. This means that there are only a small number of values that are pushing up the average. In fact, it is not hard to show that the maximal values of a regular sequence have density zero in each fundamental region.

The following proposition is a generalisation of a result of Coons and Spiegelhofer \cite[Prop.~2.3.20]{CS2018} originally proved in the special case where $f$ is the Stern sequence.

\begin{proposition}\label{prop:smallgrowth}
Let $k\geqslant 2$ be an integer and $f$ be a $k$-regular sequence. Suppose that the spectral radius $\rho({\bf B})$ is the unique dominant eigenvalue of ${\bf B}$, that $$\rho=\rho({\bf B})>\rho^*(\{{\bf B}_0,\ldots,{\bf B}_{k-1}\})=\rho^*>\rho/k,$$ that for $n$ large enough $\varSigma(n)\neq 0$ and that the asymptotical behaviour of $\varSigma(n)$ is determined by $\rho({\bf B})$. For each $m\geqslant 0$, let $g_m$ be the function defined on $[0,1]$ by $$g_m(x)=\frac 1{(\rho^*)^m}\cdot f\bigl(k^m+\bigl\lfloor k^m(k-1)x\bigr\rfloor\bigr).$$ The sequence $\{g_m(x)\}_{m\geqslant 0}$ of functions converges to zero for $\lambda$-almost all $x$ in $[0,1]$.
\end{proposition}

\begin{proof}  
By assumption, the sum of $f$ over the interval $\bigl[k^m,k^{m+1}\bigr)$ is asymptotic to $c_1\rho^m$ for some $c_1>0$. Set $M_m:=\max_{n\in[k^m,k^{m+1})}f(n)$. To prove the proposition, we need to show that there are exponentially few integers $n$ in $\bigl[k^m,k^{m+1}\bigr)$ such that $f(n)\geqslant \varepsilon M_m$,
for any given $\varepsilon>0$. By the nonnegativity of $f$, the number $N$ of such integers satisfies $N\varepsilon M_m\leqslant c_1\rho^m$, therefore $$N\leqslant c_1\rho^m/(M_m\varepsilon)\ll (\rho/\rho^*)^m/\varepsilon.$$
Since $\rho^*>\rho/k$, there are exponentially few integers $n$ such that $f(n)$ is large; in particular, there is a $K<1$ such that
$\lambda(\{x\in[0,1]:g_m(x)\geqslant \varepsilon\})\leqslant K^m/\varepsilon.$
It follows that
\begin{multline*}
\lambda\bigl(\{x\in[0,1]:\exists m\geqslant M\textrm{ such that }g_m(x)\geqslant \varepsilon\}\bigr)\\ 
=\lambda\left(\bigcup_{m\geqslant M}\{x\in[0,1]:g_m(x)\geqslant \varepsilon\}\right) 
\leqslant \sum_{m\geqslant M}\lambda\bigl(\{x\in[0,1]:g_m(x)\geqslant \varepsilon\}\bigr)\\ 
\leqslant \frac 1{\varepsilon}\sum_{m\geqslant M}K^m
=\frac 1\varepsilon\cdot\frac{K^M}{1-K}.
\end{multline*}
Thus
$$\lambda\bigl(\{x\in[0,1]:g_m(x)<\varepsilon\textrm{ for all }m\geqslant M\}\bigr)\geqslant 1-\frac{1}{\varepsilon}\cdot\frac{K^M}{1-K},$$
so that
\begin{align*}
1=\lambda\Biggl(\bigcup_{M\geqslant 1}\{x\in[0,1]:g_m(x)<\varepsilon\textrm{ for all }m\geqslant M\}\Biggr)=\lambda(A_\varepsilon),
\end{align*}
where $A_\varepsilon=\{x\in[0,1]:\exists M\geqslant 1\textrm{ such that $g_m(x)<\varepsilon$ for all $m\geqslant M$}\}.$
Therefore
\begin{equation*}
\lambda\bigl(\{x\in[0,1]:g_m(x)\rightarrow 0\textrm{ as }m\rightarrow\infty\}\bigr)
=\lambda\Biggl(\bigcap_{\varepsilon>0}A_\varepsilon\Biggr)
=\lambda\Biggl(\bigcap_{n\geqslant 1}A_{1/n}\Biggr)
=1.\qedhere
\end{equation*}
\end{proof}

The following result is an immediate corollary.

\begin{corollary}\label{cor:smallgrowth} For Lebesgue-almost all $(x)_2=0.x_1x_2x_3\cdots\in[0,1]$, we have $$\lim_{n\to\infty}\frac{1}{(\rho^*)^n}\cdot\|{\bf B}_{x_1}{\bf B}_{x_2}\cdots{\bf B}_{x_n}\|=0.$$
\end{corollary}

\noindent Unfortunately, the result of Corollary~\ref{cor:smallgrowth} is not strong enough to replace \eqref{eq:agm}. 

Note that in Proposition \ref{prop:smallgrowth}, we assumed that $\rho^*>\rho/k$. The borderline case when $\rho^{\ast}=\rho/k$ is interesting. In particular, this implies that the lower local dimension $\underline{\textnormal{dim}}(\mu,x)$ of the ghost measure $\mu$ is $1$ for $\mu$-almost every $x\in \mathbb{T}$; compare \cite[Cor.~3]{CEM}. An example is the Josefus sequence, which has an absolutely continuous ghost measure; see \cite{Epre}. But, having lower local dimension $1$ does not suffice to conclude that $\mu$ is absolutely continuous; see for example \cite{Varju} in the context of Bernoulli convolutions. 

As a final comment, we draw the reader to the similarities of the ghost distributions of the Salem sequence with digits $(b_0,b_1)=(2,3)$ and the Zaremba sequence $z_2$; see Figure \ref{fig:SandZ}. Of course, even a very careful reader could look at the two distributions in Figure \ref{fig:SandZ} and believe them the same---such a reader may be forgiven. In fact, these curves are remarkably close to one another---note the scale on the graphs. Their difference is plotted as a function of $x$ in Figure \ref{fig:SminusZ} (right) alongside the difference of their associated attractors (left). Exact agreement occurs at the points $(0, 0)$, $(1/2, 2/5)$  and $(1, 1)$ by design, and there appears to be two more points of equality, roughly at $x=5/12$ and $x=2/3$.

\begin{figure}[ht]
\begin{center}
  \includegraphics[width=.49\linewidth]{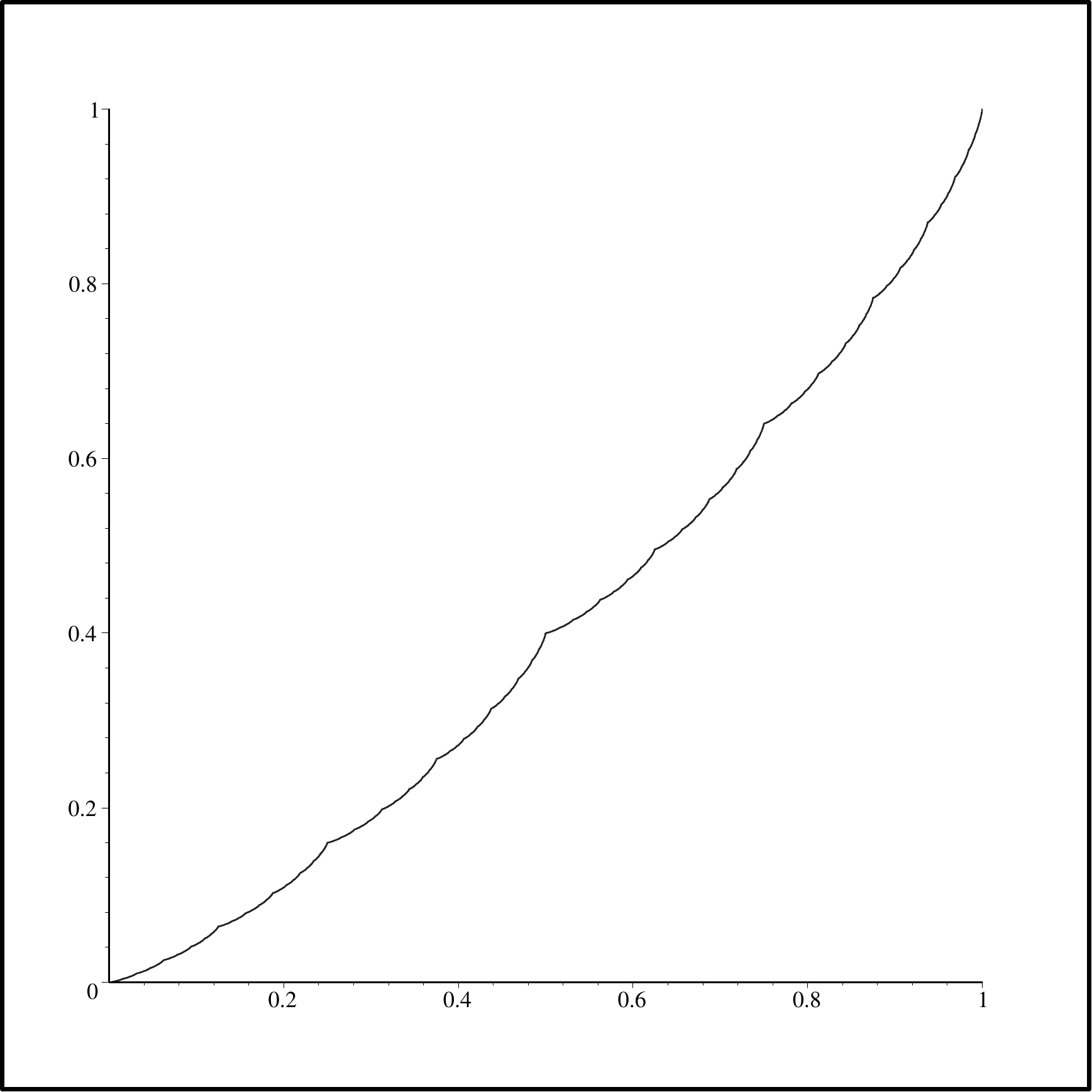}
  \includegraphics[width=.49\linewidth]{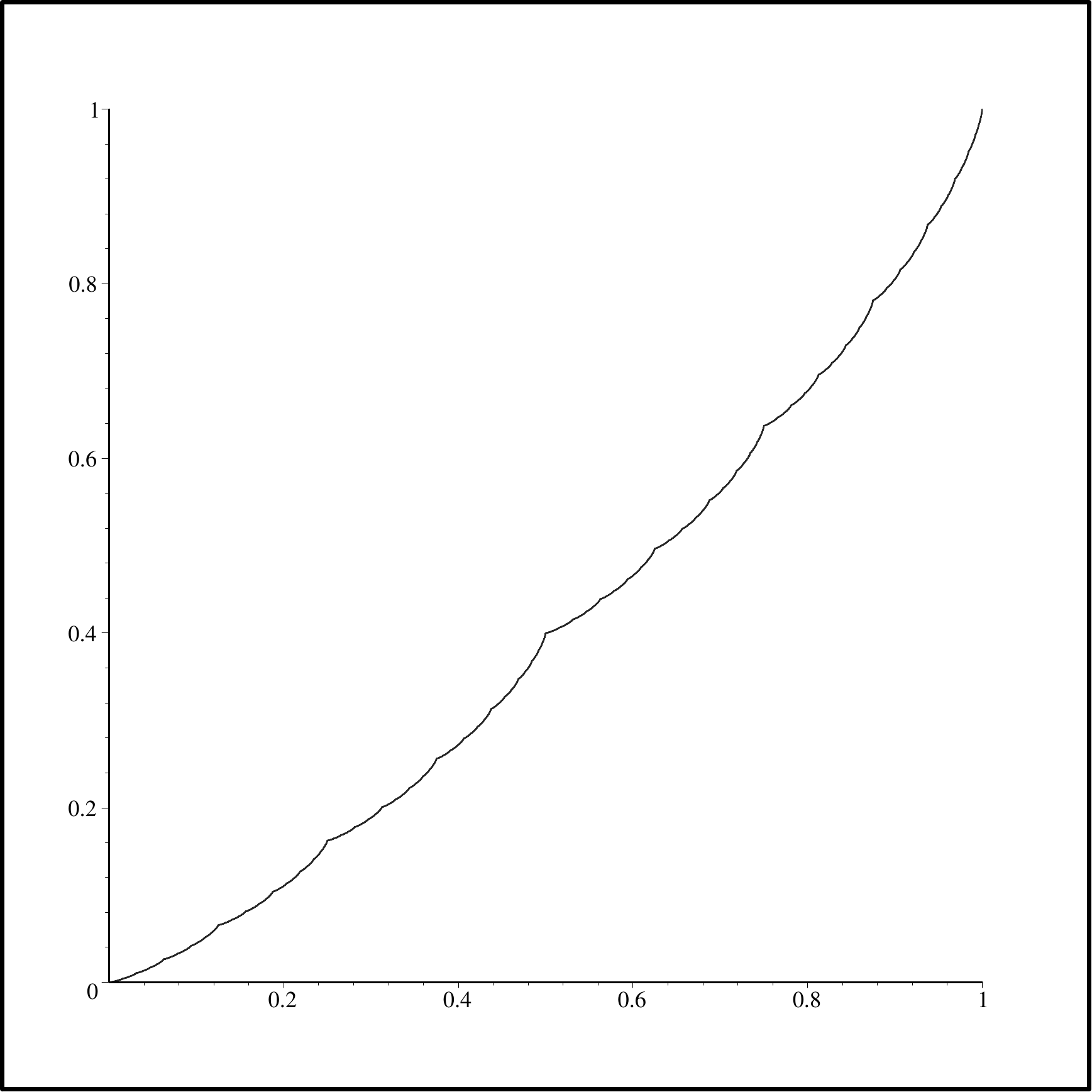} 
  \end{center}
\caption{\small The ghost distributions of the $2$-regular Salem sequence with $(b_0,b_1)=(2,3)$ (left) and the Zaremba sequence $z_2$ (right), which, while different, look remarkably similar to the naked eye.}
\label{fig:SandZ}
\end{figure}


\begin{figure}[ht]
\begin{center}
  \includegraphics[width=.49\linewidth]{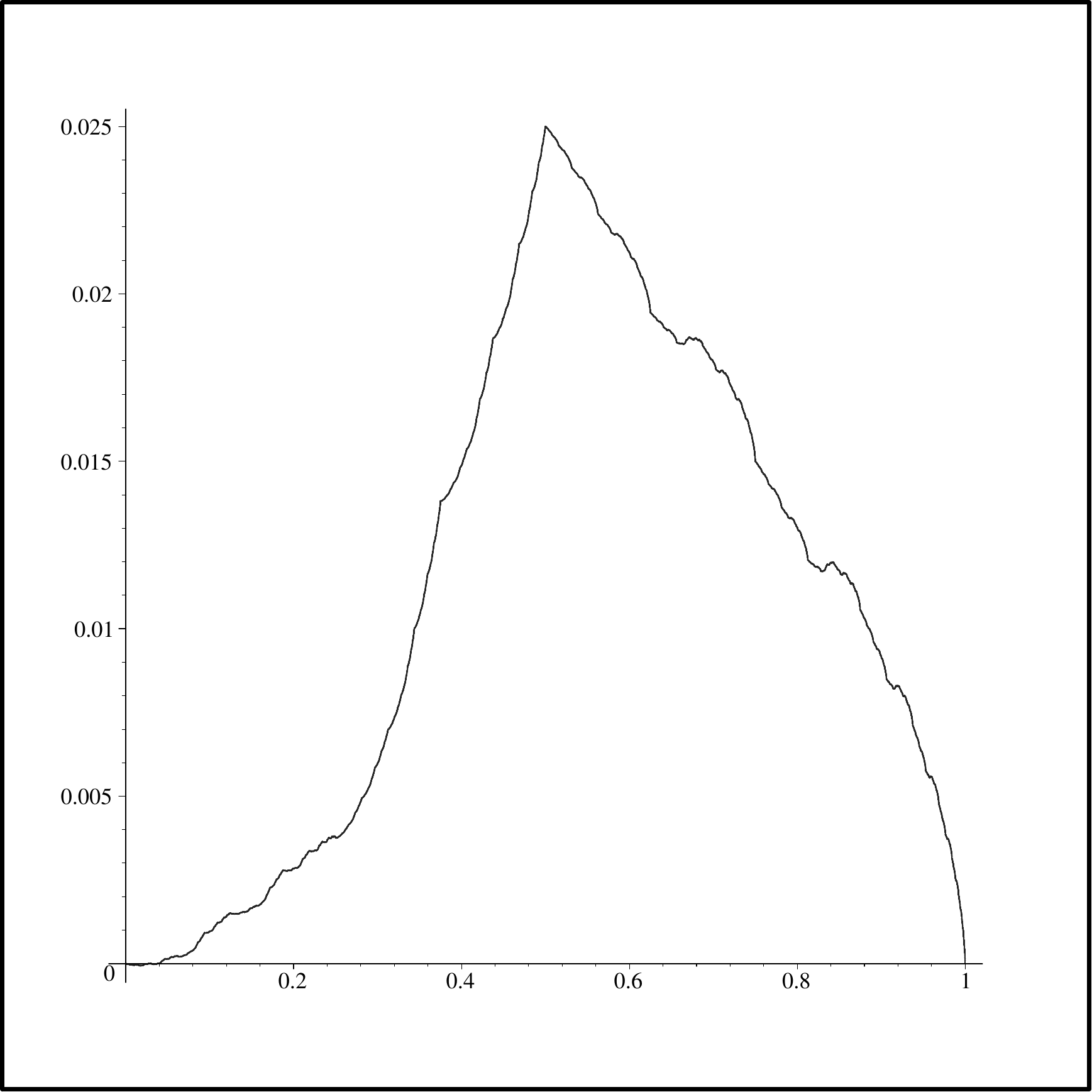}
  \includegraphics[width=.49\linewidth]{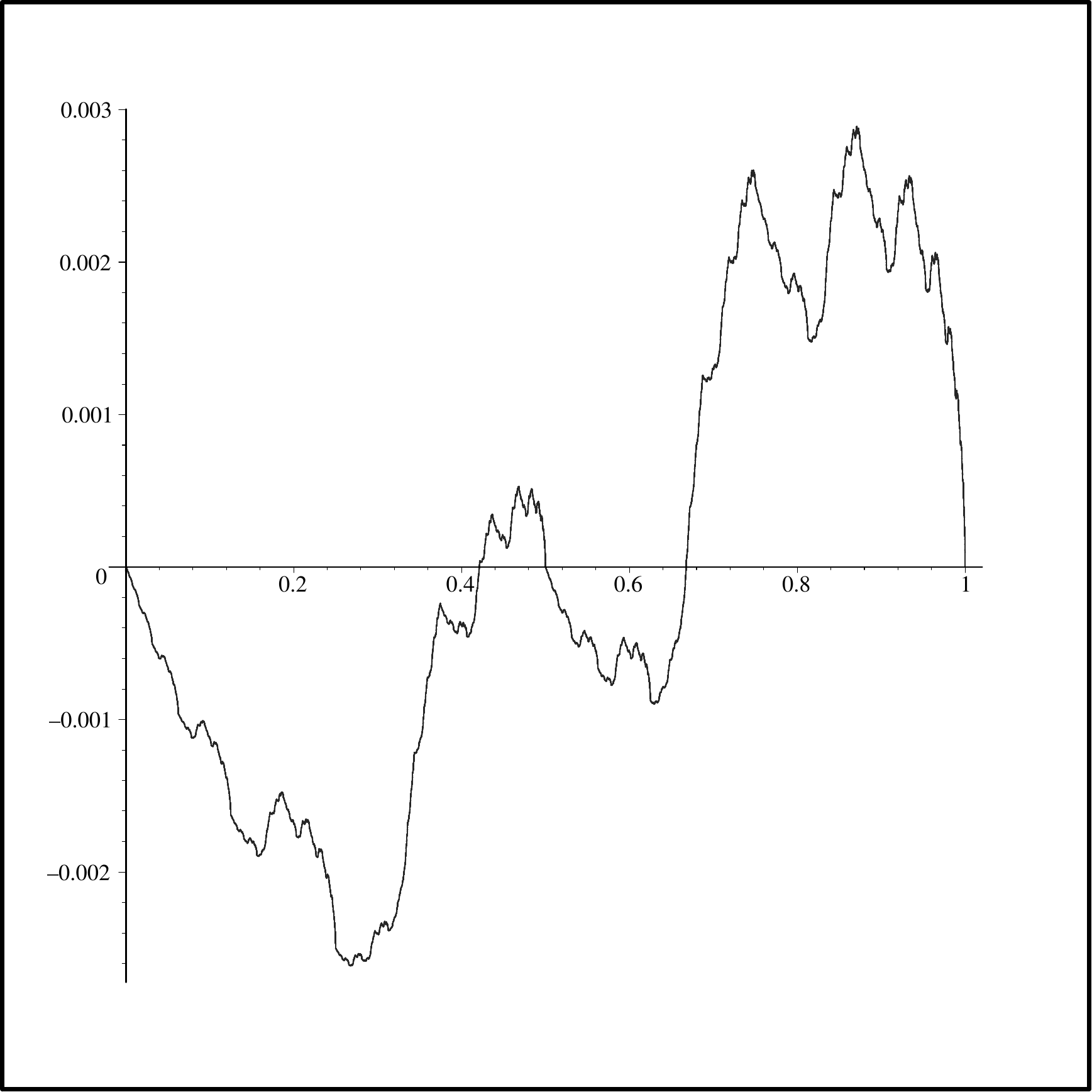}  
  \end{center}
\caption{The difference between the first coordinates of the dilation equation solutions (left) and the ghost measures (right) of the $2$-regular Salem sequence with $(b_0,b_1)=(2,3)$ and the Zaremba sequence $z_2$ (Salem minus Zaremba).}
\label{fig:SminusZ}
\end{figure}

\bibliographystyle{amsplain}

\begin{thebibliography}{10}

\bibitem{AS1992}
Jean-Paul Allouche and Jeffrey Shallit, \emph{The ring of {$k$}-regular
  sequences}, Theoret. Comput. Sci. \textbf{98} (1992), no.~2, 163--197.

\bibitem{BvdPSZ}
Jonathan Borwein, Alf van~der Poorten, Jeffrey Shallit, and Wadim Zudilin,
  \emph{Neverending fractions}, Australian Mathematical Society Lecture Series,
  vol.~23, Cambridge University Press, Cambridge, 2014.

\bibitem{BK2014}
Jean Bourgain and Alex Kontorovich, \emph{On {Z}aremba's conjecture}, Ann. of
  Math. (2) \textbf{180} (2014), no.~1, 137--196.

\bibitem{C2018}
Michael Coons, \emph{Mahler takes a regular view of {Z}aremba}, Integers
  \textbf{18A} (2018), Paper No. A6, 15.

\bibitem{CE2021}
Michael Coons and James Evans, \emph{A sequential view of self-similar
  measures; or, what the ghosts of {M}ahler and {C}antor can teach us about
  dimension}, J. Integer Seq. \textbf{24} (2021), no.~2, 21.2.5--10.

\bibitem{CEM}
Michael Coons, James Evans, and Neil Ma\~nibo, \emph{The spectral theory of
  regular sequences}, preprint.

\bibitem{CS2018}
Michael Coons and Lukas Spiegelhofer, \emph{Number theoretic aspects of regular
  sequences}, Sequences, groups, and number theory, Trends Math.,
  Birkh\"{a}user/Springer, 2018, pp.~37--87.

\bibitem{DL1991}
Ingrid Daubechies and Jeffrey~C. Lagarias, \emph{Two-scale difference
  equations. {I}. {E}xistence and global regularity of solutions}, SIAM J.
  Math. Anal. \textbf{22} (1991), no.~5, 1388--1410.

\bibitem{DL1992}
\bysame, \emph{Two-scale difference equations. {II}. {L}ocal regularity,
  infinite products of matrices and fractals}, SIAM J. Math. Anal. \textbf{23}
  (1992), no.~4, 1031--1079.

\bibitem{D2013}
Philippe Dumas, \emph{Joint spectral radius, dilation equations, and asymptotic
  behavior of radix-rational sequences}, Linear Algebra Appl. \textbf{438}
  (2013), no.~5, 2107--2126.

\bibitem{D2014}
\bysame, \emph{Asymptotic expansions for linear homogeneous divide-and-conquer
  recurrences: algebraic and analytic approaches collated}, Theoret. Comput.
  Sci. \textbf{548} (2014), 25--53.

\bibitem{E1974}
Samuel Eilenberg, \emph{Automata, languages, and machines,~{V}ol.~{A}},
  Academic Press, New York, 1974.

\bibitem{Epre}
James Evans, \emph{The ghost measures of affine regular sequences}, Houston J.
  Math., to appear.

\bibitem{Fbook}
Kenneth Falconer, \emph{Fractal geometry}, third ed., John Wiley \& Sons, Ltd.,
  Chichester, 2014.

\bibitem{FK1960}
Hillel Furstenberg and Harry Kesten, \emph{Products of random matrices}, Ann.
  Math. Statist. \textbf{31} (1960), 457--469.

\bibitem{H2015}
ShinnYih Huang, \emph{An improvement to {Z}aremba's conjecture}, Geom. Funct.
  Anal. \textbf{25} (2015), no.~3, 860--914.

\bibitem{JP2019}
Oliver Jenkinson and Mark Pollicott, \emph{Rigorous dimension estimates for
  {C}antor sets arising in {Z}aremba theory}, Dynamics: topology and numbers,
  Contemp. Math., vol. 744, 2020, pp.~83--107.

\bibitem{J2009}
Rapha{\"e}l Jungers, \emph{The joint spectral radius}, Lecture Notes in Control
  and Information Sciences, vol. 385, Springer-Verlag, Berlin, 2009.

\bibitem{N1996}
Kumiko Nishioka, \emph{Mahler functions and transcendence}, Lecture Notes in
  Mathematics, vol. 1631, Springer-Verlag, Berlin, 1996.

\bibitem{RS1960}
Gian-Carlo Rota and Gilbert Strang, \emph{A note on the joint spectral radius},
  Nederl. Akad. Wetensch. Proc. Ser. A 63 = Indag. Math. \textbf{22} (1960),
  379--381.

\bibitem{S1943}
Rapha\"el Salem, \emph{On some singular monotonic functions which are strictly
  increasing}, Trans. Amer. Math. Soc. \textbf{53} (1943), 427--439.

\bibitem{Varju}
P\'{e}ter~P. Varj\'{u}, \emph{On the dimension of {B}ernoulli convolutions for
  all transcendental parameters}, Ann. of Math. (2) \textbf{189} (2019), no.~3,
  1001--1011.

\bibitem{Z1972}
Stanis\l{}aw~K. Zaremba, \emph{La m\'ethode des ``bons treillis'' pour le
  calcul des int\'egrales multiples}, Applications of number theory to
  numerical analysis ({P}roc. {S}ympos., {U}niv. {M}ontreal, {M}ontreal,
  {Q}ue., 1971), Academic Press, New York, 1972, pp.~39--119.

\end{thebibliography}
\providecommand{\bysame}{\leavevmode\hbox to3em{\hrulefill}\thinspace}
\providecommand{\MR}{\relax\ifhmode\unskip\space\fi MR }
\providecommand{\MRhref}[2]{%
  \href{http://www.ams.org/mathscinet-getitem?mr=#1}{#2}
}
\providecommand{\href}[2]{#2}


\end{document}